\documentclass[11pt]{article}
\usepackage{amsmath,amsthm,amsfonts}
\usepackage{alg,natbib,setspace,multirow,multicol,enumitem,epsfig,datetime}
\usepackage{smpl-math}

\newcommand{\rmB}{\mathrm{B}}
\begin{document}
\date{January 13, 2012}
%
%\usdate
%
%\mdyyyydate
%\newtimeformat{campmtime}{%
%  \ifthenelse{\value{HOURXII}=0}{12}{\THEHOURXII}%
%  \timeseparator\twodigit\THEMINUTE
%  \ifthenelse{\value{HOUR}<12}{\amname}{\pmname}} 
%\date{DRAFT:\: \today\ \campmtime}
%
\title{Separable Concave Optimization Approximately Equals Piecewise-Linear
  Optimization\footnote{This research is based on the second author's Ph.D.
    thesis at the Massachusetts Institute of Technology
    \cite{nostd:stratila-phd}. An extended abstract of this research has
    appeared in \cite{MR2144589}. }}
\author{Thomas L. Magnanti\footnote{School of Engineering and Sloan School
    of Management, Massachusetts Institute of Technology, 77 Massachusetts
    Avenue, Room 32-D784, Cambridge, MA 02139. E-mail: {\tt
      magnanti@mit.edu.}}
\and Dan Stratila\footnote{Rutgers Center for Operations Research and
  Rutgers Business School, Rutgers University, 640 Bartholomew Road, Room
  107, Piscataway, NJ 08854. E-mail: {\tt dstrat@rci.rutgers.edu.}}}
\begin{singlespace}
\maketitle
\end{singlespace}
{\abstract{
\begin{singlespace}
We study the problem of minimizing a nonnegative separable concave function
over a compact feasible set. We approximate this problem to within a factor
of $1+\epsilon$ by a piecewise-linear minimization problem over the same
feasible set. Our main result is that when the feasible set is a polyhedron,
the number of resulting pieces is polynomial in the input size of the
polyhedron and linear in $1/\epsilon$. For many practical concave cost
problems, the resulting piecewise-linear cost problem can be formulated as a
well-studied discrete optimization problem. As a result, a variety of
polynomial-time exact algorithms, approximation algorithms, and
polynomial-time heuristics for discrete optimization problems immediately
yield fully polynomial-time approximation schemes, approximation algorithms,
and polynomial-time heuristics for the corresponding concave cost problems.

We illustrate our approach on two problems. For the concave cost
multicommodity flow problem, we devise a new heuristic and study its
performance using computational experiments. We are able to approximately
solve significantly larger test instances than previously possible, and
obtain solutions on average within 4.27\% of optimality. For the concave
cost facility location problem, we obtain a new $1.4991 + \epsilon$
approximation algorithm.
\end{singlespace}
}}
%
%
% __________________________________________________________________________
% Text
%
\section{Introduction}
Minimizing a nonnegative separable concave function over a polyhedron arises
frequently in fields such as transportation, logistics, telecommunications,
and supply chain management. In a typical application, the polyhedral
feasible set arises due to network structure, capacity requirements, and
other constraints, while the concave costs arise due to economies of scale,
volume discounts, and other practical factors \cite[see
  e.g.][]{MR91k:90211}. The concave functions can be nonlinear,
piecewise-linear with many pieces, or more generally given by an oracle.

A natural approach for solving such a problem is to approximate each concave
function by a piecewise-linear function, and then reformulate the resulting
problem as a discrete optimization problem. Often this transformation can be
carried out in a way that preserves problem structure, making it possible to
apply existing discrete optimization techniques to the resulting problem. A
wide variety of techniques is available for these problems, including
heuristics \cite[e.g.][]{MR90h:90061,MR1663414}, integer programming methods
\cite[e.g.][]{MR1859158,MR1970120}, and approximation algorithms
\cite[e.g.][]{MR2146253}.

For this approach to be efficient, we need to be able to approximate the
concave cost problem by a single piecewise-linear cost problem that meets
two competing requirements. On one hand, the approximation should employ few
pieces so that the resulting problem will have small input size. On the
other hand, the approximation should be precise enough that by solving the
resulting problem we would obtain an acceptable approximate solution to the
original problem.

With this purpose in mind, we introduce a method for approximating a concave
cost problem by a piecewise-linear cost problem that provides a $1+\epsilon$
approximation in terms of optimal cost, and yields a bound on the number of
resulting pieces that is polynomial in the input size of the feasible
polyhedron and linear in $1/\epsilon$. Previously, no such polynomial bounds
were known, even if we allow any dependence on $1/\epsilon$.

Our bound implies that polynomial-time exact algorithms, approximation
algorithms, and polynomial-time heuristics for many discrete optimization
problems immediately yield fully polynomial-time approximation schemes,
approximation algorithms, and polynomial-time heuristics for the
corresponding concave cost problems. We illustrate this result by obtaining
a new heuristic for the concave cost multicommodity flow problem, and a new
approximation algorithm for the concave cost facility location problem.

Under suitable technical assumptions, our method can be generalized to
efficiently approximate the objective function of a maximization or
minimization problem over a general feasible set, as long as the objective
is nonnegative, separable, and concave. In fact, our approach is not limited
to optimization problems. It is potentially applicable for approximating
problems in dynamic programming, algorithmic game theory, and other settings
where new solutions methods become available when switching from concave to
piecewise-linear functions.
\subsection{Previous Work}
Piecewise-linear approximations are used in a variety of contexts in science
and engineering, and the literature on them is expansive. Here we limit
ourselves to previous results on approximating a separable concave function
in the context of an optimization problem.

Geoffrion \cite{MR0448879} obtains several general results on approximating
objective functions. One of the settings he considers is the minimization of
a separable concave function over a general feasible set. He derives
conditions under which a piecewise-linear approximation of the objective
achieves the smallest possible absolute error for a given number of pieces.

Thakur \cite{MR491449} considers the maximization of a separable concave
function over a convex set defined by separable constraints. He approximates
both the objective and constraint functions, and bounds the absolute error
when using a given number of pieces in terms of feasible set parameters, the
maximum values of the first and second derivatives of the functions, and
certain dual optimal solutions.

Rosen and Pardalos \cite{MR838476} consider the minimization of a quadratic
concave function over a polyhedron. They reduce the problem to a separable
one, and then approximate the resulting univariate concave functions. The
authors derive a bound on the number of pieces needed to guarantee a given
approximation error in terms of objective function and feasible polyhedron
parameters. They use a nonstandard definition of approximation error,
dividing by a scale factor that is at least the maximum minus the minimum of
the concave function over the feasible polyhedron. Pardalos and Kovoor
\cite{MR1054141} specialize this result to the minimization of a quadratic
concave function over one linear constraint subject to upper and lower
bounds on the variables.

G\"uder and Morris \cite{MR1300824} study the maximization of a separable
quadratic concave function over a polyhedron. They approximate the objective
functions, and bound the number of pieces needed to guarantee a given
absolute error in terms of function parameters and the lengths of the
intervals on which the functions are approximated.

Kontogiorgis \cite{MR1795786} also studies the maximization of a separable
concave function over a polyhedron. He approximates the objective functions,
and uses techniques from numerical analysis to bound the absolute error when
using a given number of pieces in terms of the maximum values of the second
derivatives of the functions and the lengths of the intervals on which the
functions are approximated.

Each of these prior results differs from ours in that, when the goal is to
obtain a $1+\epsilon$ approximation, they do not provide a bound on the
number of pieces that is polynomial in the input size of the original
problem, even if we allow any dependence on $1/\epsilon$.

Meyerson et al. \cite{MR1931859} remark, in the context of the single-sink
concave cost multicommodity flow problem, that a ``tight'' approximation
could be computed. Munagala \cite{munagala-phd} states, in the same context,
that an approximation of arbitrary precision could be obtained with a
polynomial number of pieces. They do not mention specific bounds, or any
details on how to do so.

Hajiaghayi et al. \cite{MR1991628} and Mahdian et al. \cite{MR2247734}
consider the unit demand concave cost facility location problem, and employ
an exact reduction by interpolating the concave functions at points $1, 2,
\dots, m$, where $m$ is the number of customers. The size of the resulting
problem is polynomial in the size of the original problem, but the approach
is limited to problems with unit demand.
\subsection{Our Results}
In Section \ref{sect:gen-feas}, we introduce our piecewise-linear
approximation approach, on the basis of a minimization problem with a
compact feasible set in $\bbR^n_+$ and a nonnegative separable concave
function that is nondecreasing. In this section, we assume that the problem
has an optimal solution $x^* = (x^*_1, \dots, x^*_n)$ with $x^*_i \in \{0\}
\cup [l_i, u_i]$ and $0< l_i\le u_i$. To obtain a $1+\epsilon$
approximation, we need only $1 + \bigl\lceil \log_{1 + 4\epsilon +
  4\epsilon^2} \frac{u_i}{l_i} \bigr\rceil$ pieces for each concave
function. As $\epsilon \to 0$, the number of pieces behaves as
$\frac{1}{4\epsilon}\log \frac{u_i}{l_i}$.

In Section \ref{sect:low-bnd}, we show that any piecewise-linear approach
requires at least $\Omega\bigl(\frac{1}{\sqrt{\epsilon}} \log
\frac{u_i}{l_i}\bigr)$ pieces to approximate a certain function to within
$1+\epsilon$ on $[l_i, u_i]$. Note that for any fixed $\epsilon$, the number
of pieces required by our approach is within a constant factor of this lower
bound. It is an interesting open question to find tighter upper and lower
bounds on the number of pieces as $\epsilon\to 0$. In Section
\ref{sect:gen-feas:ext}, we extend our approximation approach to objective
functions that are not monotone and feasible sets that are not contained in
$\bbR^n_+$.

In Sections \ref{sect:poly-feas} and \ref{sect:poly-feas:ext}, we obtain the
main result of this paper. When the feasible set is a polyhedron and the
cost function is nonnegative separable concave, we can obtain a
$1+\epsilon$ approximation with a number of pieces that is polynomial in the
input size of the feasible polyhedron and linear in $1/\epsilon$. We first
obtain a result for polyhedra in $\mathbb{R}^n_+$ and nondecreasing cost
functions in Section \ref{sect:poly-feas}, and then derive the general
result in Section \ref{sect:poly-feas:ext}.

In Section \ref{sect:algor}, we show how our piecewise-linear approximation
approach can be combined with algorithms for discrete optimization problems
to obtain new algorithms for problems with concave costs. We use a
well-known integer programming formulation that often enables us to write
piecewise-linear problems as discrete optimization problems in a way that
preserves problem structure.

In Section \ref{sect:flow}, we illustrate our method on the concave cost
multicommodity flow problem. We derive considerably smaller bounds on the
number of required pieces than in the general case. Using the formulation
from Section \ref{sect:algor}, the resulting discrete optimization problem
can be written as a fixed charge multicommodity flow problem. This enables
us to devise a new heuristic for concave cost multicommodity flow by
combining our piecewise-linear approximation approach with a dual ascent
method for fixed charge multicommodity flow due to Balakrishnan et al.
\cite{MR90h:90061}.

In Section \ref{sect:flow:comput}, we report on computational experiments.
The new heuristic is able to solve large-scale test problems to within
4.27\% of optimality, on average. The concave cost problems have up to 80
nodes, 1,580 edges, 6,320 commodities, and 9,985,600 variables. These
problems are, to the best of our knowledge, significantly larger than
previously solved concave cost multicommodity flow problems, whether
approximately or exactly. A brief review of the literature on concave cost
flows can be found in Sections \ref{sect:flow} and \ref{sect:flow:comput}.

In Section \ref{sect:loc}, we illustrate our method on the concave cost
facility location problem. Combining a 1.4991-approximation algorithm for
the classical facility location problem due to Byrka \cite{Byrka:2007:OBA}
with our approach, we obtain a $1.4991+\epsilon$ approximation algorithm for
concave cost facility location. Previously, the lowest approximation ratio
for this problem was that of a $3+\epsilon$ approximation algorithm due to
Mahdian and Pal \cite{MR2085471}. In the second author's Ph.D. thesis
\cite{nostd:stratila-phd}, we obtain a number of other algorithms for
concave cost problems, including a 1.61-approximation algorithm for concave
cost facility location. Independently, Romeijn et al. \cite{MR2598822}
developed 1.61 and 1.52-approximation algorithms for this problem. A brief
review of the literature on concave cost facility location can be found in
Section \ref{sect:loc}.
\section{General Feasible Sets}
\label{sect:gen-feas}
We examine the general concave cost minimization problem
\begin{equation}
Z^*_{\ref{mp:conc:gen}} = \min \left \{ \phi(x) : x \in X \right\},
\label{mp:conc:gen}
\end{equation}
defined by a compact feasible set $X \subseteq \bbR^n_+$ and a nondecreasing
separable concave function $\phi:\bbR^n_+\to \bbR_+$. Let $x=(x_1, \dots,
x_n)$ and $\phi(x) = \sum_{i=1}^n \phi_i(x_i)$, and assume that the
functions $\phi_i$ are nonnegative. The feasible set need not be convex or
connected---for example, it could be the feasible set of an integer program.

In this section, we impose the following technical assumption. Let
$[n]=\{1,\dots,n\}$.
\begin{assumption}
\label{assum:gen}
The problem has an optimal solution $x^*=(x^*_1, \dots, x^*_n)$ and bounds
$l_i$ and $u_i$ with $0<l_i\le u_i$ such that $x^*_i\in \{0\} \cup [l_i,
  u_i]$ for $i\in [n]$.
\end{assumption}

Let $\epsilon>0$. To approximate problem \eqref{mp:conc:gen} to within a
factor of $1+\epsilon$, we approximate each function $\phi_i$ by a
piecewise-linear function $\psi_i : \bbR_+ \to \bbR_+$. Each function
$\psi_i$ consists of $P_i+1$ pieces, with $P_i = \bigl \lceil
\log_{1+\epsilon} \frac{u_i}{l_i} \bigr\rceil$, and is defined by the
coefficients
\begin{subequations}
\label{eq:pieces-def}
\begin{align}
s_i^p &= \phi'_i\left(l_i (1+\epsilon)^p\right),
&p\in\{0,\dots,P_i\},\\
f_i^p &= \phi_i\left(l_i (1+\epsilon)^p\right) - l_i (1+\epsilon)^p s_i^p,
&p\in\{0,\dots,P_i\}.
\end{align}
\end{subequations}
If the derivative $\phi'_i\left(l_i (1+\epsilon)^p\right)$ does not exist,
we take the right derivative, that is $s_i^p = \lim_{x_i \to (l_i (1+
  \epsilon)^p)^+} \frac{\phi_i(l_i (1+ \epsilon)^p) - \phi_i(x_i)} {l_i (1+
  \epsilon)^p - x_i}$. The right derivative always exists at points in
$(0,+\infty)$ since $\phi_i$ is concave on $[0, +\infty)$. We proceed in
  this way throughout the paper when the derivative does not exist.

Each coefficient pair $(s^p_i, f^p_i)$ defines a line with nonnegative slope
$s_i^p$ and y-intercept $f_i^p$, which is tangent to the graph of $\phi_i$
at the point $l_i(1+\epsilon)^p$. For $x_i > 0$, the function $\psi_i$ is
defined by the lower envelope of these lines:
\begin{equation}
\label{eq:pwl-def}
\psi_i(x_i) = \min\{ f_i^p + s_i^p x_i : p=0,\dots,P_i \}.
\end{equation}
We let $\psi_i(0) = \phi_i(0)$ and $\psi(x) = \sum_{i=1}^n \psi_i(x_i)$.
Substituting $\psi$ for $\phi$ in problem \eqref{mp:conc:gen}, we obtain the
piecewise-linear cost problem
\begin{equation}
\label{mp:pwl:gen}
Z^*_{\ref{mp:pwl:gen}} = \min \{ \psi(x) : x\in X\}.
\end{equation}

Next, we prove that this problem provides a $1+\epsilon$ approximation for
problem \eqref{mp:conc:gen}. The following proof has an intuitive
geometric interpretation, but does not yield a tight analysis of the
approximation guarantee. A tight analysis will follow.
\begin{lemma}
\label{lm:gen}
$Z^*_{\ref{mp:conc:gen}} \le Z^*_{\ref{mp:pwl:gen}}
\le (1+\epsilon) Z^*_{\ref{mp:conc:gen}}$.
\end{lemma}
\begin{proof}
Let $x^*=(x^*_1, \dots, x^*_n)$ be an optimal solution to problem
\eqref{mp:pwl:gen}; an optimal solution exists since $\psi(x)$ is concave
and $X$ is compact. Fix an $i\in [n]$, and note that for each $p\in \{0,
\dots, P_i\}$, the line $f_i^p + s^p_i x_i$ is tangent from above to the
graph of $\phi_i(x_i)$. Hence $\phi_i(x^*_i) \le \min \{ f_i^p + s_i^p x^*_i
: p = 0, \dots, P_i \} = \psi_i(x^*_i)$. Therefore, $Z^*_{\ref{mp:conc:gen}}
\le \phi(x^*) \le \psi(x^*) = Z^*_{\ref{mp:pwl:gen}}$.

Conversely, let $x^*$ be an optimal solution of problem
\eqref{mp:conc:gen} that satisfies Assumption \ref{assum:gen}. It
suffices to show that $\psi_i(x^*_i) \le (1+\epsilon) \phi_i(x^*_i)$ for
$i\in [n]$. If $x^*_i = 0$, then the inequality holds. Otherwise, let $p =
\bigl\lfloor \log_{1+\epsilon} \frac{x^*_i}{l_i} \bigr\rfloor$, and note
that $p\in \{0,\dots,P_i\}$ and $\frac{x^*_i}{l_i} \in [(1+\epsilon)^p,
  (1+\epsilon)^{p+1})$. Because $\phi_i$ is concave, nonnegative, and
  nondecreasing,
\begin{equation}
\begin{split}
\psi_i(x^*_i) & \le f_i^p + s_i^p x^*_i 
\le f_i^p + s_i^p l_i (1+\epsilon)^{p+1} \\
& = f_i^p + s_i^p l_i (1+\epsilon) (1+\epsilon)^p
\le (1+\epsilon) \left(f_i^p + s_i^p l_i (1+\epsilon)^p \right) \\
& = (1+\epsilon) \phi_i\left( l_i (1+\epsilon)^p\right) 
\le (1+\epsilon) \phi_i(x^*_i).
\end{split}
\end{equation}
(See Figure \ref{fig:intuit} for an illustration.) Therefore,
$Z^*_{\ref{mp:pwl:gen}} \le \psi(x^*) \le (1+\epsilon) \phi(x^*) =
(1+\epsilon) Z^*_{\ref{mp:conc:gen}}$.
\end{proof}
\begin{figure}[t]
\begin{center}
\epsfig{file=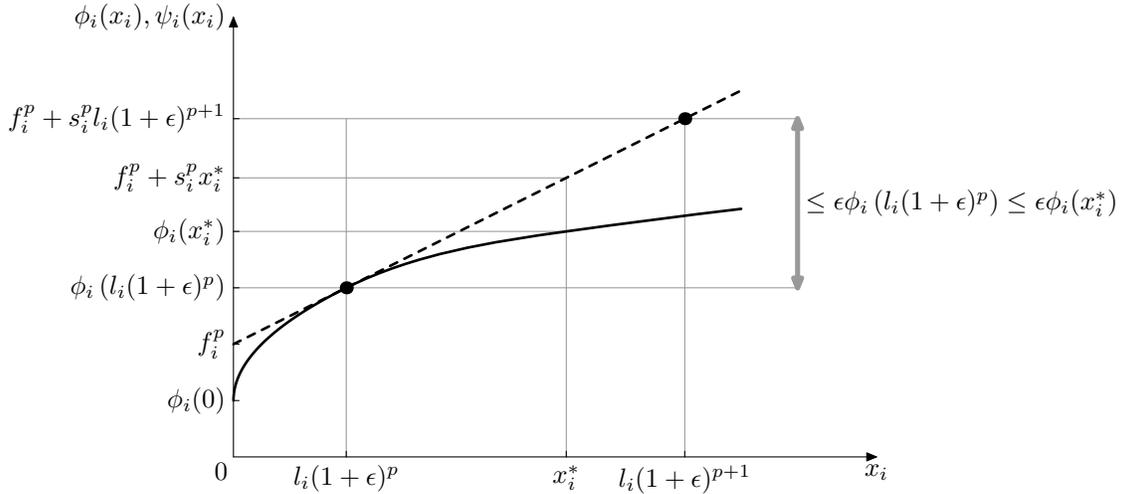}
\end{center}
\caption{Illustration of the proof of Lemma \ref{lm:gen}. Observe that the
  height of any point inside the box with the bold lower left and upper
  right corners exceeds the height of the box's lower left corner by at most
  a factor of $\epsilon$.
\label{fig:intuit}}
\end{figure}
We now present a tight analysis.

\begin{theorem}
$Z^*_{\ref{mp:conc:gen}} \le Z^*_{\ref{mp:pwl:gen}} \le
  \frac{1+\sqrt{\epsilon+1}}{2} Z^*_{\ref{mp:conc:gen}}$. The
  approximation guarantee of $\frac{1+\sqrt{\epsilon+1}}{2}$ is tight.
\label{th:gen-tight}
\end{theorem}
\begin{proof}
We have shown that $Z^*_{\ref{mp:conc:gen}} \le Z^*_{\ref{mp:pwl:gen}}$ in
Lemma \ref{lm:gen}. Fix an $i\in [n]$, and consider the approximation ratio
achieved on $[l_i, u_i]$ when approximating $\phi_i$ by $\psi_i$. If
$\phi_i(l_i) = 0$, then $\phi_i(x_i) = 0$ for all $x_i \ge 0$, and we have a
trivial case. If $\phi_i(l_i) > 0$, then $\phi_i(x_i) > 0$ for all $x_i >
0$, and the ratio is
\begin{multline}
\label{eq:approx-ratio}
\quad\qquad \min \{ 1+ \gamma : \psi_i(x_i) \le 
(1+\gamma) \phi_i( x_i) \text{ for } x_i \in [l_i, u_i] \} \\
= \max\{ \psi_i(x_i) / \phi_i(x_i) : x_i\in [l_i, u_i] \}. \qquad\quad
\end{multline}
We derive an upper bound of $\frac{1+\sqrt{\epsilon+1}}{2}$ on this ratio,
and then construct a family of functions that, when taken as $\phi_i$, yield
ratios converging to this upper bound.

Without loss of generality, we assume $l_i=1$ and $u_i =
(1+\epsilon)^{P_i}$. The approximation ratio achieved on $[1, u_i]$ is the
highest of the approximation ratios on each of the intervals $[1,
  1+\epsilon], \dots, [(1+\epsilon)^{P_i-1}, (1+\epsilon)^{P_i}]$. By
scaling along the x-axis, it is enough to consider only the interval
$[1,1+\epsilon]$, and therefore we can assume that $\psi_i$ is given by the
two tangents to the graph of $\phi_i$ at $1$ and $1+\epsilon$. Suppose these
tangents have slopes $a$ and $c$ respectively. We can assume that
$\phi_i(0)=0$, and that $\phi_i$ is linear with slope $a$ on $[0,1]$ and
linear with slope $c$ on $[1+\epsilon, +\infty)$. By scaling along the
  y-axis, we can assume that $a=1$.

We upper bound the approximation ratio between $\psi_i$ and $\phi_i$ by the
ratio between $\psi_i$ and a new function $\varphi_i$ that has
$\varphi_i(0)=0$ and consists of 3 linear pieces with slopes $1 \ge b \ge c$
on $[0,1]$, $[1, 1+\epsilon]$, and $[1+\epsilon, +\infty)$ respectively. The
approximation ratio between $\psi_i$ and $\varphi_i$ can be viewed as a
function of $b$ and $c$. Let $1+\xi$ be a point on the interval
$[1,1+\epsilon]$. We are interested in the following maximization problem
with respect to $b$, $c$, and $\xi$:
\begin{equation}
\max \{ \psi_i(1+\xi) / \varphi_i(1+\xi) : 1\ge b\ge c\ge 0, 0\le \xi\le
\epsilon \}.
\end{equation}
Since $\varphi_i$ consists of 3 linear pieces, while $\psi_i$ is given by
the lower envelope of two tangents, we have
\begin{equation}
\varphi_i(1+\xi)=1+b\xi, \qquad \psi_i(x_i)=\min\left\{ 1+\xi, 1+b\epsilon -
c(\epsilon-\xi)\right\}.
\end{equation}
(See Figure \ref{fig:tight} for an illustration.)

Since $\xi \le \epsilon$, we can assume that $c=0$. Next, since we seek to
find $b$ and $\xi$ that maximize
\begin{equation}
\frac{\psi_i(1+\xi)}{\varphi_i(1+\xi)} =
\frac{\min\{1+\xi,1+b\epsilon\}}{1+b\xi} = \min\left\{ \frac{1+\xi}{1+b\xi},
\frac{1+b\epsilon}{1+b\xi}\right\},
\end{equation}
we can assume $\xi$ is such that $\frac{1+\xi}{1+b\xi} =
\frac{1+b\epsilon}{1+b\xi}$, which yields $\xi = \epsilon b$. Substituting,
we now seek to maximize $\frac{1+\epsilon b}{1+\epsilon b^2}$, and we find
that the maximum is achieved at $b = \frac{1}{1+\sqrt{\epsilon+1}}$ and
equals $\frac{1+\sqrt{\epsilon+1}}{2}$. Therefore,
$\frac{1+\sqrt{\epsilon+1}}{2}$ is an approximation guarantee for our
approach.

Finally, we show that this guarantee is tight. First, let $\phi_i$ be the
new function $\varphi_i$ with $b$ and $c$ taken to have the values that
yield the guarantee above. If we were to use our approach to approximate
$\phi_i$, the tangents at $1$ and $1+\epsilon$ would have slopes $b$ and
$c$, instead of the desired $1$ and $c$, since $\phi_i$ lacks derivatives at
$1$ and $1+\epsilon$, and our approach uses the derivative from the right
when the derivative does not exist. We can cause our approach to produce
tangents at $1$ and $1+\epsilon$ with slopes $1$ and $c$ by taking a
sufficiently small $\zeta$ and letting $\phi_i$ have its first breakpoint at
$1+\zeta$ instead of $1$. When $\zeta\to 0$, the approximation ratio
achieved by our approach for $\phi_i$ converges to the guarantee above.
\end{proof}
\begin{figure}[t]
\begin{center}
\epsfig{file=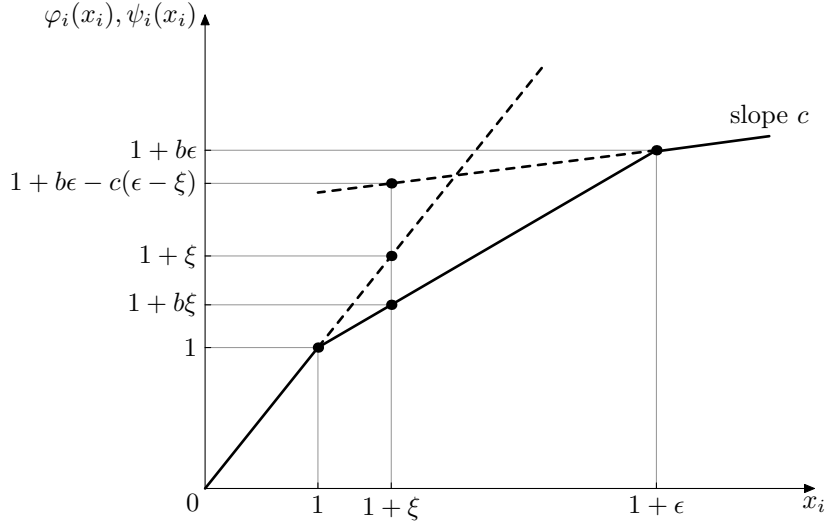}
\end{center}
\caption{Illustration of the proof of Theorem \ref{th:gen-tight}.
\label{fig:tight}}
\end{figure}
To compare the tight approximation guarantee with that provided by Lemma
\ref{lm:gen}, note that $\frac{1+\sqrt{\epsilon+1}}{2} \le
1+\frac{\epsilon}{4}$ for $\epsilon > 0$. Moreover, since
$\frac{1+\sqrt{\epsilon+1}}{2} \to 1$ and
$\frac{d}{d\epsilon}\frac{1+\sqrt{\epsilon+1}}{2} \to \frac{1}{4}$ as
$\epsilon\to 0$, it follows that $1+\frac{\epsilon}{4}$ is the lowest ratio
of the form $1+\frac{\epsilon}{k}$ that is guaranteed by our approach.

Equivalently, instead of an approximation guarantee of
$\frac{1+\sqrt{\epsilon+1}}{2}$ using $1 + \bigl\lceil \log_{1+\epsilon}
\frac{u_i}{l_i} \bigr\rceil$ pieces, we can obtain a guarantee of
$1+\epsilon$ using only $1 + \bigl\lceil \log_{1 + 4\epsilon + 4\epsilon^2}
\frac{u_i}{l_i} \bigr\rceil$ pieces. Note that $\log_{1 + 4\epsilon +
  4\epsilon^2} \frac{u_i}{l_i} = \frac{1}{\log (1+4\epsilon + 4\epsilon^2)}
\log \frac{u_i}{l_i}$, and as $\epsilon \to 0$, we have $\frac{1}{\log(1 +
  4\epsilon + 4\epsilon^2)} \to +\infty$ and
$\frac{4\epsilon}{\log(1+4\epsilon+4\epsilon^2)} \to 1$. Therefore, as
$\epsilon\to 0$, the number of pieces behaves as $\frac{1}{4\epsilon} \log
\frac{u_i}{l_i}$.

This bound on the number of pieces enables us to apply our approximation
approach to practical concave cost problems. In Section
\ref{sect:poly-feas}, we will exploit the logarithmic dependence on
$\frac{u_i}{l_i}$ of this bound to derive polynomial bounds on the number of
pieces for problems with polyhedral feasible sets.

\subsection{A Lower Bound on the Number of Pieces}
\label{sect:low-bnd}
Since the approximation guarantee of Theorem \ref{th:gen-tight} is tight, $1
+ \bigl\lceil \log_{1 + 4\epsilon + 4\epsilon^2} \frac{u_i}{l_i}
\bigr\rceil$ is a lower bound on the number of pieces needed to guarantee a
$1+\epsilon$ approximation when using the approach of equations
\eqref{eq:pieces-def}--\eqref{mp:pwl:gen}. In this section, we establish a
lower bound on the number of pieces needed to guarantee a $1+\epsilon$
approximation when using \emph{any} piecewise-linear approximation approach.

First, we show that by limiting ourselves to approaches that use
piecewise-linear functions whose pieces are tangent to the graph of the
original function, we increase the number of needed pieces by at most a
factor of 3.

Let $\phi_i:\bbR_+\to\bbR_+$ be a nondecreasing concave function, which we
are interested in approximating on an interval $[l_i, u_i]$ with $0<l_i\le
u_i$. Assume that $\phi_i(x_i) > 0$ for all $x_i > 0$; if $\phi_i(x_i) = 0$
for some $x_i > 0$, then $\phi_i$ must be zero everywhere on $[0, +\infty)$,
  and we have a trivial case. Also let $\psi_i:\bbR_+ \to \bbR_+$ be a
  piecewise-linear function with $Q$ pieces that approximates $\phi_i$ on
  $[l_i, u_i]$ to within a factor of $1+\epsilon$, that is $\frac{1}{1 +
    \epsilon} \le \frac{\psi_i(x_i)}{\phi_i(x_i)} \le 1 + \epsilon$ for
  $x_i\in [l_i,u_i]$. We are not imposing any other assumptions on $\psi_i$;
  in particular it need not be continuous, and its pieces need not be
  tangent to the graph of $\phi_i$.
\begin{lemma}
The function $\phi_i$ can be approximated on $[l_i, u_i]$ to within a factor
of $1+\epsilon$ by a piecewise-linear function $\varphi_i:\bbR_+ \to \bbR_+$
that has at most $3Q$ pieces and whose every piece is tangent to the graph
of $\phi_i$.
\label{lm:low-bnd:reduct}
\end{lemma}
\begin{proof}
First, we translate the pieces of $\psi_i$ that are strictly above $\phi_i$
down, and the pieces strictly below up, until they intersect $\phi_i$. Let
the modified function be $\psi'_i$; clearly $\psi'_i$ still provides a
$1+\epsilon$ approximation for $\phi_i$ on $[l_i, u_i]$.

For each piece of $\psi'_i$, we proceed as follows. Let $f_i^p + s_i^p x_i$
be the line defining the piece, and $[a_p, b_p]$ be the interval covered by
the piece on the x-axis. Without loss of generality, assume that $[a_p, b_p]
\subseteq [l_i, u_i]$. If the piece is tangent to $\phi_i$, we take it as
one of the pieces composing $\varphi_i$, ensuring that $\varphi_i$ provides
a $1+\epsilon$ approximation for $\phi_i$ on $[a_p, b_p]$.

If the piece is not tangent, it must intersect $\phi_i$ at either one or two
points. If the piece intersects at two points $\xi_1$ and $\xi_2$, then the
points partition the interval $[a_p, b_p]$ into three subintervals: $[a_p,
  \xi_1]$, on which the piece is above $\phi_i$; $[\xi_1, \xi_2]$, on which
the piece is below $\phi_i$; and $[\xi_2, b_p]$, on which the piece is again
above $\phi_i$. If there is one intersection point, we can partition $[a_p,
  b_p]$ similarly, except that one or two of the subintervals would be
empty.

On the interval $[a_p, \xi_1]$, the line $f_i^p + s_i^p x_i$ is above
$\phi_i$ and provides a $1+\epsilon$ approximation for $\phi_i$. We take the
tangent to $\phi_i$ at $\xi_1$ as one of the pieces composing $\varphi_i$,
ensuring that $\varphi_i$ provides a $1+\epsilon$ approximation for $\phi_i$
on $[a_p, \xi_1]$. Similarly, we take the tangent to $\phi_i$ at $\xi_2$ as
one of the pieces, ensuring a $1+\epsilon$ approximation on $[\xi_2, b_p]$.

Next, note that on the interval $[\xi_1, \xi_2]$, the line $f_i^p +s_i^p
x_i$ is below $\phi_i$ and provides a $1+\epsilon$ approximation for
$\phi_i$. Therefore, on $[\xi_1, \xi_2]$, the scaled line $(1+\epsilon)f_i^p
+ (1+\epsilon)s_i^p x_i$ is above $\phi_i$ and still provides a $1+\epsilon$
approximation for $\phi_i$. Since the original line is below and the scaled
line above $\phi_i$, there is an $\epsilon^*$ with $0 < \epsilon^* \le
\epsilon$ such that, on $[\xi_1, \xi_2]$, the line $(1+\epsilon^*)f_i^p +
(1+\epsilon^*)s_i^p x_i$ is above $\phi_i$ and intersects it at one or more
points. If this line is tangent, we take it as one of the pieces that define
$\varphi_i$. If the line is not tangent, it must intersect $\phi_i$ at
exactly one point $\xi'$, and we take the tangent to $\phi_i$ at $\xi'$ as
one of the pieces. In either case, we have ensured that $\varphi_i$ provides
a $1+\epsilon$ approximation for $\phi_i$ on $[\xi_1, \xi_2]$.

Since $\cup_{p=1}^Q [a_p, b_p] = [l_i, u_i]$, the constructed function
$\varphi_i$ provides a $1+\epsilon$ approximation for $\phi_i$ on $[l_i,
  u_i]$. Since for each piece of $\psi_i$, we introduced at most 3 pieces,
$\varphi_i$ has at most $3Q$ pieces.
\end{proof}
Next, we establish a lower bound on the number of pieces needed to
approximate the square root function to within a factor of $1+\epsilon$ by a
piecewise-linear function that has its every piece tangent to the graph of
the original function. Let $\phi_i(x_i)=\sqrt{x_i}$, and let $\varphi_i$ be
a piecewise-linear function that approximates $\phi_i$ to within a factor of
$1+\epsilon$ on $[l_i, u_i]$ and whose every piece is tangent to the graph
of $\phi_i$.

To write the lower bounds in this section in a more intuitive way, we define
the function $\gamma(\epsilon) = \bigl(1 + 2\epsilon(2 + \epsilon) +
2(1+\epsilon) \sqrt{\epsilon(2 + \epsilon)}\bigr)^2$. As $\epsilon \to 0$,
$\gamma(\epsilon)$ behaves as $1+\sqrt{32\epsilon}$, with the other terms
vanishing because they contain higher powers of $\epsilon$. In particular,
$1+\sqrt{32\epsilon} \le \gamma(\epsilon) \le 1+16\sqrt{\epsilon}$ for $0 <
\epsilon \le \frac{1}{10}$.
\begin{lemma}
The function $\varphi_i$ must contain at least $\bigl\lceil
\log_{\gamma(\epsilon)} \frac{u_i}{l_i} \bigr\rceil$ pieces. As $\epsilon\to
0$, this lower bound behaves as $\frac{1}{\sqrt{32\epsilon}} \log
\frac{u_i}{l_i}$.
\label{lm:low-bnd:tan}
\end{lemma}
\begin{proof}
Given a point $\xi_0 \in [l_i,u_i]$, a tangent to $\phi_i$ at $\xi_0$
guarantees a $1+\epsilon$ approximation on an interval extending to the left
and right of $\xi_0$. Let us denote this interval by $[\xi_0(1 + \delta_1),
  \xi_0(1 + \delta_2)]$. The values of $\delta_1$ and $\delta_2$ can be
found by solving with respect to $\delta$ the equation
\begin{subequations}
\begin{align}
\phi_i(\xi_0) + \delta \xi_0 \phi'_i (\xi_0) &= 
(1+\epsilon) \phi_i ((1+\delta)\xi_0) \\
\Leftrightarrow
\sqrt{\xi_0} + \delta \xi_0 \frac{1}{2\sqrt{\xi_0}} &= 
(1+\epsilon) \sqrt{ (1+\delta)\xi_0} \\
\Leftrightarrow
\xi_0 + \delta \xi_0 + \frac{1}{4} \delta^2 \xi_0  &= 
(1+\epsilon)^2 (1+\delta) \xi_0.
\end{align}
\end{subequations}
This is simply a quadratic equation with respect to $\delta$, and solving it
yields $\delta_1 =2 \epsilon(2+\epsilon) - 2(1+\epsilon)\sqrt{\epsilon
  (2+\epsilon)}$ and $\delta_2 = 2\epsilon(2+\epsilon) +
2(1+\epsilon)\sqrt{\epsilon (2+\epsilon)}$. Let $\xi_1=\xi_0 (1+\delta_1)$,
and note that $[\xi_0(1 + \delta_1), \xi_0(1 + \delta_2)] = \bigl[ \xi_1,
  \frac{1+\delta_2}{1+\delta_1} \xi_1 \bigr]$. Therefore, the tangent
provides a $1+\epsilon$ approximation on the interval
\begin{equation}
\left[\xi_1,\frac{1+\delta_2} {1+\delta_1} \xi_1\right] = 
\left[\xi_1, \left(1 + 2\epsilon(2 + \epsilon) +
2(1+\epsilon) \sqrt{\epsilon(2 + \epsilon)} \right)^2 \xi_1\right] =
[\xi_1,\gamma(\epsilon)\xi_1].
\end{equation}

Since $\gamma(\epsilon)$ does not depend on $\xi_1$, the best way to obtain
a $1+\epsilon$ approximation on $[l_i, u_i]$ is to iteratively introduce
tangents that provide approximations on intervals of the form $[l_i,
  \gamma(\epsilon) l_i]$, $[\gamma(\epsilon) l_i, \gamma^2(\epsilon) l_i],
[\gamma^2(\epsilon) l_i, \gamma^3(\epsilon) l_i], \dots$, until the entire
interval $[l_i, u_i]$ is covered. It immediately follows that we need at
least $\bigl\lceil \log_{\gamma(\epsilon)} \frac{u_i}{l_i} \bigr\rceil$
pieces to approximate $\phi_i$ on $[l_i,u_i]$.

This bound can also be written as $\bigl\lceil \frac{1}{\log
  \gamma(\epsilon)} \log \frac{u_i}{l_i} \bigr\rceil$. As $\epsilon\to 0$,
we have $\frac{1}{\log \gamma(\epsilon)} \to +\infty$ and
$\frac{\sqrt{32\epsilon}}{\log \gamma(\epsilon)} \to 1$, and therefore, the
lower bound behaves as $\frac{1}{\sqrt{32\epsilon}} \log \frac{u_i}{l_i}$.
\end{proof}
Combining Lemmas \ref{lm:low-bnd:reduct} and \ref{lm:low-bnd:tan}, we
immediately obtain a lower bound for any piecewise-linear approximation
approach. Let $\psi_i:\bbR_+ \to \bbR_+$ be a piecewise-linear function that
approximates $\phi_i(x_i) = \sqrt{x_i}$ to within a factor of $1+\epsilon$
on $[l_i, u_i]$. Note that $\psi_i$ need not be continuous or have its
pieces tangent to the graph of $\phi_i$.
\begin{theorem}
The function $\psi_i$ must contain at least $\bigl\lceil \frac{1}{3}
\log_{\gamma(\epsilon)} \frac{u_i}{l_i} \bigr\rceil$ pieces. As $\epsilon\to
0$, this lower bound behaves as $\frac{1}{\sqrt{288\epsilon}} \log
\frac{u_i}{l_i}$.
\end{theorem}
This lower bound is within a factor of $2+ \frac{3 \log
  \gamma(\epsilon)}{\log(1+4\epsilon+4\epsilon^2)}$ of the number of pieces
required by our approach. This implies that for fixed $\epsilon$, the number
of pieces required by our approach is within a constant factor of the best
possible. As $\epsilon \to 0$, the number of pieces needed by our approach
converges to a factor of $\frac{\sqrt{288\epsilon}}{4\epsilon} = O\bigl(
\frac{1}{\sqrt{\epsilon}} \bigr)$ of the lower bound. An interesting open
question is to find tighter upper and lower bounds on the number of pieces
as $\epsilon\to 0$.
\subsection{Extensions}
\label{sect:gen-feas:ext}
Our approximation approach applies to a broader class of problems. In this
section, we generalize our results to objective functions that are not
monotone and feasible sets that are not contained in $\bbR^n_+$. Consider
the problem
\begin{equation}
\label{mp:conc:gen-x}
Z^*_{\ref{mp:conc:gen-x}} = \min \{ \phi(x) : x \in X \},
\end{equation}
defined by a compact feasible set $X\subseteq \bbR^n$ and a separable
concave function $\phi:Y\to \bbR_+$. The feasible set $X$ need not be convex
or connected, and the set $Y$ can be any convex set in $\bbR^n$ that
contains $X$. Let $\phi(x) = \sum_{i=1}^n \phi_i(x_i)$, and assume that the
functions $\phi_i$ are nonnegative.

Instead of Assumption \ref{assum:gen}, we impose the following assumption.
Let $\proj_{x_i} Y$ denote the projection of $Y$ on $x_i$, and note that
$\proj_{x_i} Y$ is the domain of $\phi_i$.
\begin{assumption}
\label{assum:gen-x}
Problem \eqref{mp:conc:gen-x} has an optimal solution $x^* = (x^*_1, \dots,
x^*_n)$, bounds $\alpha_i, \beta_i$ with $[\alpha_i, \beta_i] \subseteq
\proj_{x_i} Y$, and bounds $l_i, u_i$ with $0 < l_i \le u_i$ such that
$x^*_i \in \{\alpha_i, \beta_i\} \cup \bigl( [\alpha_i + l_i, \alpha_i +
  u_i] \cap [\beta_i - u_i, \beta_i - l_i] \bigr)$ for $i\in [n]$.
\end{assumption}

Next, we apply the approach of equations
\eqref{eq:pieces-def}--\eqref{mp:pwl:gen} to approximate problem
\eqref{mp:conc:gen-x} to within a factor of $1+\epsilon$. We approximate
each concave function $\phi_i$ by a piecewise-linear function $\psi_i$.
Assume that the interval $[\alpha_i + l_i, \alpha_i + u_i] \cap [\beta_i -
  u_i, \beta_i - l_i]$ is nonempty; if this interval is empty, we have a
trivial case. For convenience, we define a new pair of bounds
\begin{equation}
l'_i = \max \{ l_i, \beta_i - u_i - \alpha_i \},
\qquad
u'_i = \min \{ u_i, \beta_i - l_i - \alpha_i\}. 
\label{eq:gen-x:lu}
\end{equation}
Note that $[\alpha_i + l_i, \alpha_i + u_i] \cap [\beta_i - u_i, \beta_i -
  l_i] = [\alpha_i + l'_i, \alpha_i + u'_i] = [\beta_i - u'_i, \beta_i -
  l'_i]$. Since $\phi_i$ is concave, there is a point $\xi^* \in [\alpha_i,
  \beta_i]$ such that $\phi_i$ is nondecreasing on $[\alpha_i, \xi^*]$ and
nonincreasing on $[\xi^*, \beta_i]$. We do not have to compute $\xi^*$ in
order to approximate $\phi_i$. Instead, we simply introduce tangents
starting from $\alpha_i + l'_i$ and advancing to the right, and starting
from $\beta_i - l'_i$ and advancing to the left.

More specifically, we introduce tangents starting from $\alpha_i + l'_i$
only if the slope at this point is nonnegative. We introduce tangents at
$\alpha_i + l'_i, \alpha_i + l'_i (1 + 4\epsilon + 4\epsilon^2), \dots,
\alpha_i + l'_i (1 + 4\epsilon + 4\epsilon^2)^{Q_i}$, where $Q_i$ is largest
integer such that $\alpha_i + l'_i (1 + 4\epsilon + 4\epsilon^2)^{Q_i} \le
\alpha_i + u'_i$ and the slope at $\alpha_i + l'_i (1 + 4\epsilon +
4\epsilon^2)^{Q_i}$ is nonnegative.

Let $\zeta_i = \min \bigl\{ \alpha_i + u'_i, \alpha_i + l'_i (1 + 4\epsilon
+ 4\epsilon^2)^{Q_i+1} \bigr\}$. If $\phi_i$ has a nonnegative slope at
$\zeta_i$, we introduce an additional tangent at $\zeta_i$. If the slope at
$\zeta_i$ is negative, we find the largest integer $r_i$ such that the slope
at $\alpha_i + l'_i (1 + 4\epsilon + 4\epsilon^2)^{Q_i} (1 +
\epsilon)^{r_i}$ is nonnegative, and introduce an additional tangent at that
point. Since the slope is nonnegative at $\alpha_i + l'_i (1 + 4\epsilon +
4\epsilon^2)^{Q_i}$, we have $r_i\ge 0$, and since $\zeta_i \le \alpha_i +
l'_i (1 + 4\epsilon + 4\epsilon^2)^{Q_i+1} \le \alpha_i + l'_i (1 +
4\epsilon + 4\epsilon^2)^{Q_i}(1 + \epsilon)^4$, we have $r_i \le 3$. Let
the tangents introduced starting from $\alpha_i + l'_i$ have slopes $s^0_i,
\dots, s^{Q_i+1}_i$ and y-intercepts $f^0_i, \dots, f^{Q_i+1}_i$.

We introduce tangents starting from $\beta_i - l'_i$ only if the slope at
this point is nonpositive. We proceed in the same way as with the tangents
starting from $\alpha_i + l'_i$, and let these tangents have slopes
$s^{Q_i+2}_i, \dots, s^{Q_i+R_i+3}_i$ and y-intercepts $f^{Q_i+2}_i, \dots,
f^{Q_i+R_i+3}_i$. Also let $P_i = Q_i + R_i + 3$.

If $\alpha_i$ and $\beta_i$ are the endpoints of $\proj_{x_i} Y$, for $x_i
\in (\alpha_i, \beta_i)$, the function $\psi_i$ is given by
\begin{equation}
\psi_i(x_i) = \min \{ f^p_i + s^p_i x_i : p=0, ..., P_i \},
\end{equation}
while for $x_i \in \{\alpha_i, \beta_i\}$, we let $\psi_i(x_i) =
\phi_i(x_i)$.

If $\alpha_i$ and $\beta_i$ are in the interior of $\proj_{x_i} Y$, we
introduce two more tangents at $\alpha_i$ and $\beta_i$, with slopes
$s^{P_i+1}_i, s^{P_i+2}_i$ and y-intercepts $f^{P_i+1}_i, f^{P_i+2}_i$, and
let $\psi_i(x_i) = \min \{ f^p_i + s^p_i x_i : p=0, ..., P_i +2 \}$. If one
of $\alpha_i, \beta_i$ is in the interior and the other is an endpoint, we
use the corresponding approach in each case.

We now replace the objective function $\phi(x)$ in problem
\eqref{mp:conc:gen-x} with the new objective function $\psi(x) =
\sum_{i=1}^n \psi_i(x_i)$, obtaining the piecewise-linear cost problem
\begin{equation}
\label{mp:pwl:gen-x}
Z^*_{\ref{mp:pwl:gen-x}} = \min \{ \psi(x) : x \in X \}.
\end{equation}

The number of pieces used to approximate each concave function $\phi_i$ in
each direction is at most $2 + \bigl\lceil \log_{1 + 4\epsilon +
  4\epsilon^2} \frac{u_i}{l_i} \bigr\rceil$, and therefore the total number
of pieces used for each function is at most $4 + 2 \bigl\lceil \log_{1 +
  4\epsilon + 4\epsilon^2} \frac{u_i}{l_i} \bigr\rceil$. As $\epsilon \to
0$, this bound behaves as $\frac{1}{2\epsilon} \log \frac{u_i}{l_i}$.

It remains to show that problem \eqref{mp:pwl:gen-x} provides a $1+\epsilon$
approximation for problem \eqref{mp:conc:gen-x}, which we do by employing
Lemma \ref{lm:gen} and Theorem \ref{th:gen-tight}.
\begin{lemma}
$Z^*_{\ref{mp:conc:gen-x}} \le Z^*_{\ref{mp:pwl:gen-x}}
\le (1+\epsilon) Z^*_{\ref{mp:conc:gen-x}}$.
\label{lm:gen-x}
\end{lemma}
\begin{proof}
Clearly, $Z^*_{\ref{mp:conc:gen-x}} \le Z^*_{\ref{mp:pwl:gen-x}}$. To prove
the inequality's other side, let $x^*$ be an optimal solution to problem
\eqref{mp:conc:gen-x} that satisfies Assumption \ref{assum:gen-x}. We will
show that $\psi_i(x^*_i) \le (1+\epsilon) \phi_i(x^*_i)$ for $i\in [n]$. If
$x^*_i \in \{\alpha_i, \beta_i\}$ then $\psi_i(x^*_i) = \phi_i(x^*_i)$. If
$x^*_i \not\in \{\alpha_i, \beta_i\}$, we must have $x^*_i \in [\alpha_i +
  l'_i, \alpha_i + u'_i] = [\beta_i - u'_i, \beta_i - l'_i]$. Since $\phi_i$
is concave, it is nondecreasing on $[\alpha_i, x^*_i]$, nonincreasing on
$[x^*_i, \beta_i]$, or both. Without loss of generality, assume that
$\phi_i$ is nondecreasing on $[\alpha_i, x^*_i]$.

Due to the way we introduced tangents starting from $\alpha_i + l'_i$, it
follows that $x^*_i \in [ \alpha_i + l'_i, \zeta_i]$. We divide this
interval into two subintervals, $\bigl[ \alpha_i + l'_i, \alpha_i + l'_i (1
  + 4\epsilon + 4\epsilon^2)^{Q_i} \bigr]$ and $\bigl[ \alpha_i + l'_i (1 +
  4\epsilon + 4\epsilon^2)^{Q_i}, \zeta_i \bigr]$. If $x^*_i \in \bigl[
  \alpha_i + l'_i, \alpha_i + l'_i (1 + 4\epsilon + 4\epsilon^2)^{Q_i}
  \bigr]$, then since $\phi_i$ is nondecreasing on this interval,
$\psi_i(x^*_i) \le (1+\epsilon) \phi_i(x^*_i)$ follows directly from Theorem
\ref{th:gen-tight}.

If $x^*_i \in \bigl[ \alpha_i + l'_i (1 + 4\epsilon + 4\epsilon^2)^{Q_i},
  \zeta_i \bigr]$, additional steps are needed, since $\phi_i$ is not
necessarily nondecreasing on this interval. If $\phi_i$ has a nonnegative
slope at $\zeta_i$, then we introduced a tangent at $\zeta_i$, and
$\psi_i(x^*_i) \le (1+\epsilon) \phi_i(x^*_i)$ again follows from Theorem
\ref{th:gen-tight}. If the slope at $\zeta_i$ is negative, we introduced a
tangent at $\alpha_i + l'_i (1 + 4\epsilon + 4\epsilon^2)^{Q_i} (1 +
\epsilon)^{r_i}$. Since $r_i$ is the largest integer such that the slope at
$\alpha_i + l'_i (1 + 4\epsilon + 4\epsilon^2)^{Q_i} (1 + \epsilon)^{r_i}$
is nonnegative, and the slope at $x^*_i$ is also nonnegative, $x^*_i \in
\bigl[ \alpha_i + l'_i (1 + 4\epsilon + 4\epsilon^2)^{Q_i}, \alpha_i + l'_i
  (1 + 4\epsilon + 4\epsilon^2)^{Q_i} (1+\epsilon)^{r_i + 1} \bigr]$.

We now distinguish two cases. If $x^*_i \in \bigl[ \alpha_i + l'_i (1 +
  4\epsilon + 4\epsilon^2)^{Q_i}, \alpha_i + l'_i (1 + 4\epsilon +
  4\epsilon^2)^{Q_i} (1+\epsilon)^{r_i} \bigr]$, then since $\phi_i$ is
nondecreasing on this interval, $\psi_i(x^*_i) \le (1+\epsilon)
\phi_i(x^*_i)$ follows by Theorem \ref{th:gen-tight}. If $x^*_i \in \bigl[
  \alpha_i + l'_i (1 + 4\epsilon + 4\epsilon^2)^{Q_i} (1+\epsilon)^{r_i},
  \alpha_i + l'_i (1 + 4\epsilon + 4\epsilon^2)^{Q_i} (1+\epsilon)^{r_i + 1}
  \bigr]$, note that the right endpoint of this interval is $1+\epsilon$
times farther from $\alpha_i$ than the left endpoint. Since $\phi_i$ is
nondecreasing from the left endpoint to $x^*$, and we introduced a tangent
at the left endpoint, $\psi_i(x^*_i) \le (1+\epsilon) \phi_i(x^*_i)$ follows
by Lemma \ref{lm:gen}.

Taken together, the above cases imply that $Z^*_{\ref{mp:pwl:gen-x}} \le
\psi(x^*) \le (1 + \epsilon) \phi(x^*) = (1+\epsilon)
Z^*_{\ref{mp:conc:gen-x}}$.
\end{proof}
We conclude this section with two further extensions:
\begin{enumerate}[label=\arabic*)]
\item
We can use secants instead of tangents, in which case we require on the
order of one function evaluation per piece, and do not need to evaluate the
derivative. The secant approach may be preferable in computational
applications where derivatives are difficult to compute.
\item
The results in this section can be adapted to apply to concave maximization
problems.
\end{enumerate}
\section{Polyhedral Feasible Sets}
\label{sect:poly-feas}
In this section and Section \ref{sect:poly-feas:ext}, we obtain the main
result of this paper by applying our approximation approach to concave cost
problems with polyhedral feasible sets. We will employ the polyhedral
structure of the feasible set to eliminate the quantities $l_i$ and $u_i$
from the bound on the number of pieces, and obtain a bound that is
polynomial in the input size of the concave cost problem and linear in
$1/\epsilon$.

Let $X = \{ x : Ax \le b, x\ge 0\}$ be a nonempty rational polyhedron
defined by a matrix $A\in \bbQ^{m\times n}$ and a vector $b\in \bbQ^m$. Let
$\phi: \bbR^n_+ \to \bbR_+$ be a nondecreasing separable concave function,
with $\phi(x) = \sum_{i=1}^n \phi_i(x_i)$ and each function $\phi_i$
nonnegative. We consider the problem
\begin{equation}
Z^*_{\ref{mp:conc:poly}} = \min\{ \phi(x) : Ax \le b, x\ge 0\}.
\label{mp:conc:poly}
\end{equation}

Following standard practice, we define the size of rational numbers,
vectors, and matrices as the number of bits needed to represent them
\cite[see e.g.][]{MR2003b:90004}. More specifically, for an integer $r$, let
$\size(r) = 1 + \lceil \log_2(|r| + 1) \rceil$; for a rational number $r =
\frac{r_1}{r_2}$ with $r_2 > 0$, and $r_1$ and $r_2$ coprime integers, let
$\size(r) = \size(r_1) + \size(r_2)$; and for a rational vector or matrix
$M\in \bbQ^{p\times q}$ with elements $m_{ij}$, let $\size(M) = p q +
\sum_{i=1}^p \sum_{j=1}^q \size(m_{ij})$.

We take the input size of problem \eqref{mp:conc:poly} to be the input size
of the feasible polyhedron, $\size(A) + \size(b)$. Assume that each function
$\phi_i$ is given by an oracle that returns the function value $\phi_i(x_i)$
and derivative $\phi'_i(x_i)$ in time $O(1)$. When the concave functions are
given in other ways than through oracles, the input size of problem
\eqref{mp:conc:poly} is at least $\size(A) + \size(b)$, and therefore our
bound applies in those cases as well.

We will use the following classical result that bounds the size of a
polyhedron's vertices in terms of the size of the constraint matrix and
right-hand side vector that define the polyhedron \cite[see
  e.g.][]{MR2003b:90004}. Let $U(A,b) = 4(\size(A) + \size(b) + 2n^2 + 3n)$.
\begin{lemma}
If $x'=(x'_1, \dots, x'_n)$ is a vertex of $X$, then each of its components
has $\size(x'_i) \le U(A,b)$.
\label{lm:vert-size}
\end{lemma}
To approximate problem \eqref{mp:conc:poly}, we replace each concave
function $\phi_i$ with a piecewise-linear function $\psi_i$ as described in
equations \eqref{eq:pieces-def}--\eqref{mp:pwl:gen}. To obtain each function
$\psi_i$, we take
\begin{equation}
l_i = \frac{1}{2^{U(A,b)-1} - 1}, \qquad u_i = 2^{U(A,b)-1} - 1,
\label{eq:poly:lu}
\end{equation} 
and $P_i = \lceil \log_{1 + 4\epsilon + 4\epsilon^2} \frac{u_i}{l_i}
\rceil$, and introduce $P_i+1$ tangents to $\phi_i$ at $l_i, l_i(1 +
4\epsilon + 4\epsilon^2), \dots, l_i(1 + 4\epsilon + 4\epsilon^2)^{P_i}$.
The resulting piecewise-linear cost problem is
\begin{equation}
Z^*_{\ref{mp:pwl:poly}} =\min\{ \psi(x) : Ax \le b, x\ge 0\}.
\label{mp:pwl:poly}
\end{equation}

The number of pieces used to approximate each function $\phi_i$ is
\begin{equation}
1 + \Bigl\lceil \log_{1 + 4\epsilon + 4\epsilon^2} \frac{u_i}{l_i}
\Bigr\rceil \le 1 + \bigl\lceil \log_{1 + 4\epsilon + 4\epsilon^2}
2^{2U(A,b)} \bigr\rceil = 1 + \biggl\lceil \frac{2U(A,b)}{\log_2 (1 +
  4\epsilon + 4\epsilon^2)} \biggr\rceil.
\label{eq:poly:bound}
\end{equation}
As $\epsilon \to 0$, this bound behaves as $\frac{2 U(A,b)}{4 (\log_2 e)
  \epsilon} = \frac{U(A,b)}{2 (\log_2 e) \epsilon}$. Therefore, the obtained
bound is polynomial in the size of the input and linear in $1/\epsilon$. The
time needed to compute the piecewise-linear approximation is also polynomial
in the size of the input and linear in $1/\epsilon$. Specifically, we can
compute all the quantities $l_i$ and $u_i$ in $O(U(A,b))$, and then compute
the pieces composing each function $\psi_i$ in $O\bigl( \frac{
  U(A,b)}{\log_2 (1 + 4\epsilon + 4\epsilon^2)} \bigr)$ per function, for a
total running time of $O\bigl( U(A,b) + \frac{n U(A,b)}{\log_2 (1 +
  4\epsilon + 4\epsilon^2)} \bigr) = O\bigl( \frac{n U(A,b)}{\epsilon}
\bigr)$.

Next, we apply Theorem \ref{th:gen-tight} to show that problem
\eqref{mp:pwl:poly} approximates problem \eqref{mp:conc:poly} to within a
factor of $1+\epsilon$.
\begin{lemma}
$Z^*_{\ref{mp:conc:poly}} \le Z^*_{\ref{mp:pwl:poly}} 
\le (1+\epsilon) Z^*_{\ref{mp:conc:poly}}$.
\label{lm:poly}
\end{lemma}
\begin{proof} %
It is clear that problem \eqref{mp:conc:poly} satisfies the assumptions
needed by Theorem \ref{th:gen-tight}, except for Assumption \ref{assum:gen}
and the requirement that $X$ be a compact set. Next, we consider these two
assumptions.

Because $X$ is a polyhedron in $\bbR^n_+$ and $\phi$ is concave and
nonnegative, problem \eqref{mp:conc:poly} has an optimal solution $x^*$ at a
vertex of $X$ \cite{MR0131816}. Lemma \ref{lm:vert-size} ensures that
$\size(x^*_i) \le U(A,b)$ for $i\in [n]$, and hence $x^*_i \in \{ 0 \} \cup
\bigl[ \frac{1}{2^{U(A,b)-1} - 1}, 2^{U(A,b)-1} -1 \bigr]$. Therefore,
problem \eqref{mp:conc:poly} together with the bounds $l_i$ and $u_i$, and
the optimal solution $x^*$ satisfies Assumption \ref{assum:gen}.

If the polyhedron $X$ is bounded, then Theorem \ref{th:gen-tight} applies,
and the approximation property follows. If $X$ is unbounded, we add the
constraints $x_i \le 2^{U(A,b)-1} - 1$ for $i\in [n]$ to problems
\eqref{mp:conc:poly} and \eqref{mp:pwl:poly}, obtaining the modified
problems
\begin{align}
Z^*_{\ref{mp:conc:poly}\rmB} &= \min \bigl\{ \phi(x) : Ax \le b, 
0\le x\le 2^{U(A,b)-1}-1 \bigr\}, 
\tag{\ref{mp:conc:poly}B} \label{mp:conc:poly-b}\\
Z^*_{\ref{mp:pwl:poly}\rmB} &=\min \bigl\{ \psi(x) : Ax \le b, 
0\le x\le 2^{U(A,b)-1}-1 \bigr\}. 
\tag{\ref{mp:pwl:poly}B} \label{mp:pwl:poly-b}
\end{align}
%
% Don't use \eqref with \tag as it yields unusual spacing.
Denote the modified feasible polyhedron by $X_\rmB$. Since $X_\rmB \subseteq
X$ and $x^*\in X_\rmB$, it follows that $Z^*_{\mathrm{\ref{mp:conc:poly-b}}}
= Z^*_{\ref{mp:conc:poly}}$ and $x^*$ is an optimal solution to problem
(\ref{mp:conc:poly-b}). Similarly, let $y^*$ be a vertex optimal solution to
problem \eqref{mp:pwl:poly}; since $X_\rmB \subseteq X$ and $y^*\in X_\rmB$,
we have $Z^*_{\mathrm{\ref{mp:pwl:poly-b}}} = Z^*_{\ref{mp:pwl:poly}}$.

Since $X_\rmB$ is a bounded polyhedron, problem (\ref{mp:conc:poly-b}),
together with the bounds $l_i$ and $u_i$, and the optimal solution $x^*$
satisfies the assumptions needed by Theorem \ref{th:gen-tight}. When we
approximate problem (\ref{mp:conc:poly-b}) using the approach of equations
\eqref{eq:pieces-def}--\eqref{mp:pwl:gen}, we obtain problem
(\ref{mp:pwl:poly-b}), and therefore $Z^*_{\mathrm{\ref{mp:conc:poly-b}}}
\le Z^*_{\mathrm{\ref{mp:pwl:poly-b}}} \le (1 + \epsilon)
Z^*_{\mathrm{\ref{mp:conc:poly-b}}}$. The approximation property follows.
\end{proof}

Note that it is not necessary to add the constraints $x_i \le 2^{U(A,b)-1} -
1$ to problem \eqref{mp:conc:poly} or \eqref{mp:pwl:poly} when computing the
piecewise-linear approximation, as the modified problems are only used in
the proof of Lemma \ref{lm:poly}.

If the objective functions $\phi_i$ of problem \eqref{mp:conc:poly} are
already piecewise-linear, the resulting problem \eqref{mp:pwl:poly} is again
a piecewise-linear concave cost problem, but with each objective function
$\psi_i$ having at most the number of pieces given by bound
\eqref{eq:poly:bound}. Since this bound does not depend on the functions
$\phi_i$, and is polynomial in the input size of the feasible polyhedron $X$
and linear in $1/\epsilon$, our approach may be used to reduce the number of
pieces for piecewise-linear concave cost problems with a large number of
pieces.

When considering a specific application, it is often possible to use the
application's structure to derive values of $l_i$ and $u_i$ that yield a
significantly better bound on the number of pieces than the general values
of equation \eqref{eq:poly:lu}. We will illustrate this with two
applications in Sections \ref{sect:flow} and \ref{sect:loc}.
\subsection{Extensions}
\label{sect:poly-feas:ext}
Next, we generalize this result to polyhedra that are not contained in
$\bbR_+^n$ and concave functions that are not monotone. Consider the problem
\begin{equation}
Z^*_{\ref{mp:conc:poly-gen}} = \min\{ \phi(x) : Ax\le b\},
\label{mp:conc:poly-gen}
\end{equation}
defined by a rational polyhedron $X = \{ x : Ax \le b \}$ with at least one
vertex, and a separable concave function $\phi:Y\to \bbR_+$. Here $Y = \{ x:
Cx \le d \}$ can be any rational polyhedron that contains $X$ and has at
least one vertex. Let $\phi(x) = \sum_{i=1}^n \phi_i(x_i)$, and assume that
the functions $\phi_i$ are nonnegative. We assume that the input size of
this problem is $\size(A) + \size(b)$, and that the functions $\phi_i$ are
given by oracles that return the function value and derivative in time
$O(1)$.

Since, unlike problem \eqref{mp:conc:poly}, this problem does not include
the constraints $x \ge 0$, we need the following variant of Lemma
\ref{lm:vert-size} \cite[see e.g.][]{MR2003b:90004}. Let $V(A,b) =
4(\size(A) + \size(b))$.
\begin{lemma}
If $x'=(x'_1, \dots, x'_n)$ is a vertex of $X$, then each of its components
has $\size(x'_i) \le V(A,b)$.
\label{lm:vert-size-gen}
\end{lemma}

We approximate this problem by applying the approach of Section
\ref{sect:gen-feas:ext} as follows. If $\proj_{x_i} Y$ is a closed interval
$[\alpha'_i, \beta'_i]$, we let $[\alpha_i, \beta_i] = [\alpha'_i,
  \beta'_i]$; if $\proj_{x_i} Y$ is a half-line $[\alpha'_i, +\infty)$ or
  $(-\infty, \beta'_i]$, we let $[\alpha_i, \beta_i] = \bigl[\alpha'_i,
  2^{V(A,b)-1} \bigr]$ or $[\alpha_i, \beta_i] = \bigl[-2^{V(A,b)-1},
  \beta'_i \bigr]$; and if the projection is the entire real line, we let
$[\alpha_i, \beta_i] = \bigl[-2^{V(A,b)}, 2^{V(A,b)} \bigr]$.

If $\proj_{x_i} Y$ is a closed interval or a half-line, we take
\begin{equation}
\label{eq:poly-gen:lu}
l_i = \frac{1}{2^{V(A,b) + V(C,d) - 1} - 1} 
\quad\text{and}\quad
u_i = 2^{V(A,b)-1} + 2^{V(C,d) - 1} - 1,
\end{equation}
while if $\proj_{x_i} Y$ is the entire real line, we take $l_i=
2^{V(A,b)-1}$ and $u_i = 3\cdot 2^{V(A,b)-1}$. We then apply the approach of
Section \ref{sect:gen-feas:ext} as described from Assumption \ref{assum:gen-x}
onward, obtaining the piecewise-linear cost problem
\begin{equation}
Z^*_{\ref{mp:pwl:poly-gen}} =\min\{ \psi(x) : Ax \le b\}.
\label{mp:pwl:poly-gen}
\end{equation}

The number of pieces used to approximate each function $\phi_i$ is at most
\begin{equation}
\label{eq:bnd:poly-gen}
\begin{split}
4 + 2 \Bigl\lceil \log_{1 + 4\epsilon + 4\epsilon^2} \frac{u_i}{l_i} 
\Bigr\rceil 
&\le 4 + 2 \bigl\lceil \log_{1 + 4\epsilon + 4\epsilon^2} 
\bigl( 2^{V(A,b) + V(C,d)} \bigl( 2^{V(A,b)} + 2^{V(C,d)} \bigr) \bigr) 
\bigr\rceil \\
&\le 4 + 2 \bigl\lceil \log_{1 + 4\epsilon + 4\epsilon^2} 
\bigl( 2^{V(A,b) + V(C,d)} 2^{V(A,b) + V(C,d)} \bigr) 
\bigr\rceil \\
&= 4 + 2\bigg\lceil \frac{2V(A,b) + 2V(C,d)}
{\log_2(1 + 4\epsilon + 4\epsilon^2)} \biggr\rceil. 
\end{split}
\end{equation}
As $\epsilon \to 0$, this bound behaves as $\frac{V(A,b) + V(C,d)}{ (\log_2
  e) \epsilon}$. Note that, in addition to the size of the input and
$1/\epsilon$, this bound also depends on the size of $C$ and $d$. Moreover,
to analyze the time needed to compute the piecewise-linear approximation, we
have to specify a way to compute the quantities $\alpha'_i$ and $\beta'_i$.
We will return to these issues shortly.

Next, we prove that problem \eqref{mp:pwl:poly-gen} approximates problem
\eqref{mp:conc:poly-gen} to within a factor of $1+\epsilon$, by applying
Lemma \ref{lm:gen-x}.
\begin{theorem}
$Z^*_{\ref{mp:conc:poly-gen}} \le Z^*_{\ref{mp:pwl:poly-gen}} 
\le (1 + \epsilon) Z^*_{\ref{mp:conc:poly-gen}}$.
\end{theorem}
\begin{proof}
The assumptions needed by Lemma \ref{lm:gen-x} are satisfied, except for
Assumption \ref{assum:gen-x} and the requirement that $X$ be a compact set.
We address these two assumptions as follows.

First, note that problem \eqref{mp:conc:poly-gen} has an optimal solution at
a vertex $x^*$ of $X$, since $X$ is a polyhedron with at least one vertex,
and $\phi$ is concave and nonnegative \cite{MR0131816}. By Lemma
\ref{lm:vert-size-gen}, we have $\size(x^*_i) \le V(A,b)$ for $i\in [n]$,
and hence $x^*_i \in \bigl[ -2^{V(A,b) - 1} + 1, 2^{V(A,b)-1} - 1 \bigr]$.
We add the constraints $-2^{V(A,b) - 1} + 1 \le x_i \le 2^{V(A,b)-1} - 1$
for $i\in [n]$ to $X$, obtaining the polyhedron $X_\rmB$ and the problems
\begin{align}
Z^*_{\ref{mp:conc:poly-gen}\rmB} &= \min \bigl\{ \phi(x) : Ax \le b,
-2^{V(A,b) - 1} + 1 \le x \le 2^{V(A,b)-1} - 1 \bigr\},
\tag{\ref{mp:conc:poly-gen}B} \label{mp:conc:poly-gen-b}\\
Z^*_{\ref{mp:pwl:poly-gen}\rmB} &=\min \bigl\{ \psi(x) : Ax \le b,
-2^{V(A,b) - 1} + 1 \le x \le 2^{V(A,b)-1} - 1 \bigr\}.
\tag{\ref{mp:pwl:poly-gen}B} \label{mp:pwl:poly-gen-b}
\end{align}
It is easy to see that $Z^*_{\mathrm{\ref{mp:conc:poly-gen-b}}} =
Z^*_{\ref{mp:conc:poly-gen}}$ and $Z^*_{\mathrm{\ref{mp:pwl:poly-gen-b}}} =
Z^*_{\ref{mp:pwl:poly-gen}}$, and that $x^*$ is an optimal solution to
problem (\ref{mp:conc:poly-gen-b}).

Clearly, $X_\rmB$ is a compact set. To see that Assumption \ref{assum:gen-x}
is satisfied for problem (\ref{mp:conc:poly-gen-b}), consider the following
three cases:
\begin{enumerate}[label=\arabic*)]
\item
If $\proj_{x_i} Y = (-\infty, +\infty)$, then $\alpha_i = -2^{V(A,b)}$, $l_i
= 2^{V(A,b)-1}$, and $u_i = 3\cdot 2^{V(A,b)-1}$. As a result, $x^*_i -
\alpha_i \in \bigl[-2^{V(A,b)-1} + 1 + 2^{V(A,b)}, 2^{V(A,b)-1} - 1 +
  2^{V(A,b)} \bigr] \subseteq \{ 0 \} \cup [l_i, u_i]$.
\item
If $\proj_{x_i} Y = (-\infty, \beta'_i]$, then $\alpha_i = -2^{V(A,b)-1}$,
and thus $x^*_i - \alpha_i \ge 1$. On the other hand, $\beta_i = \beta'_i$,
implying that $\beta_i$ is a component of a vertex of $Y$, and thus
$\size(\beta_i) \le V(C,d)$. Now, $x^*_i \le \beta_i$ implies that $x^*_i -
\alpha_i \le 2^{V(C,d) -1} - 1 + 2^{V(A,b) - 1}$. Since $l_i =
\frac{1}{2^{V(A,b) + V(C,d) - 1} - 1}$ and $u_i = 2^{V(A,b)-1} + 2^{V(C,d) -
  1} - 1$, we have $x^*_i - \alpha_i \in \{ 0 \} \cup [l_i, u_i]$.
\item
If $\proj_{x_i} Y = [\alpha'_i, +\infty)$ or $\proj_{x_i} Y = [\alpha'_i,
    \beta'_i]$, then let $x^*_i = \frac{p_1}{q_1}$ with $q_1 > 0$, and $p_1$
  and $q_1$ coprime integers. Similarly let $\alpha_i = \frac{p_2}{q_2}$,
  and note that $x^*_i - \alpha_i = \frac{p_1 q_2 - p_2 q_1}{q_1 q_2}$.
  Since $\alpha_i = \alpha'_i$, we know that $\alpha_i$ is a component of a
  vertex of $Y$, and hence $\size(\alpha_i) \le V(C,d)$ and $\size(q_2) \le
  V(C,d)$. On the other hand, $\size(x^*_i) \le V(A,b)$, and thus
  $\size(q_1) \le V(A,b)$. This implies that $\size(q_1 q_2) \le V(A,b) +
  V(C,d)$, and therefore either $x^*_i = \alpha_i$ or $x^*_i - \alpha_i \ge
  \frac{1}{2^{V(A,b) + V(C,d) - 1} - 1}$. Next, since $\size(\alpha_i) \le
  V(C,d)$ and $\size(x^*_i) \le V(A,b)$, we have $x^*_i - \alpha_i \le
  2^{V(A,b)-1} + 2^{V(C,d)-1} - 2$. Given that $l_i = \frac{1}{2^{V(A,b) +
      V(C,d) - 1} - 1}$ and $u_i = 2^{V(A,b)-1} + 2^{V(C,d) - 1} - 1$, it
  follows that $x^*_i - \alpha_i \in \{ 0 \} \cup [l_i, u_i]$.

\end{enumerate}
Combining the three cases, we obtain $x^*_i \in \{ \alpha_i \} \cup
[\alpha_i + l_i, \alpha_i + u_i]$. Similarly, we can show that $x^*_i \in \{
\beta_i \} \cup [\beta_i - u_i, \beta_i - l_i]$, and therefore $x^*_i \in
(\{ \alpha_i \} \cup [\alpha_i + l_i, \alpha_i + u_i]) \cap (\{ \beta_i \}
\cup [\beta_i - u_i, \beta_i - l_i])$, which is a subset of $\{ \alpha_i,
\beta_i \} \cup \bigl( [\alpha_i + l_i, \alpha_i + u_i] \cap [\beta_i - u_i,
  \beta_i - l_i] \bigr)$.

Therefore, problem (\ref{mp:conc:poly-gen-b}), together with the quantities
$\alpha_i$ and $\beta_i$, the bounds $l_i$ and $u_i$, and the optimal
solution $x^*$ satisfies Assumption \ref{assum:gen-x}, and Lemma
\ref{lm:gen-x} applies. Using the approach of Section
\ref{sect:gen-feas:ext} to approximate problem (\ref{mp:conc:poly-gen-b})
yields problem (\ref{mp:pwl:poly-gen-b}), which implies that
$Z^*_{\mathrm{\ref{mp:conc:poly-gen-b}}} \le
Z^*_{\mathrm{\ref{mp:pwl:poly-gen-b}}} \le (1+\epsilon)
Z^*_{\mathrm{\ref{mp:conc:poly-gen-b}}}$, and the approximation property
follows.
\end{proof}
To obtain a bound on the number of pieces that is polynomial in the size of
the input and linear in $1/\epsilon$, we can simply restrict the domain $Y$
of the objective function to the feasible polyhedron $X$, that is let $Y :=
X$. In this case, bound \eqref{eq:bnd:poly-gen} becomes $4 + 2\big\lceil
\frac{4V(A,b)} {\log_2(1 + 4\epsilon + 4\epsilon^2)} \bigr\rceil$, and can
be further improved to $4 + 2\bigl\lceil \frac{3V(A,b)} {\log_2(1 +
  4\epsilon + 4\epsilon^2)} \bigr\rceil$, which behaves as $\frac{1.5
  V(A,b)}{(\log_2 e) \epsilon}$ as $\epsilon \to 0$.

When $Y=X$, the time needed to compute the piecewise-linear approximation is
also polynomial in the size of the input and linear in $1/\epsilon$. The
quantities $\alpha'_i$ and $\beta'_i$ can be computed by solving the linear
programs $\min \{ x_i : Ax \le b\}$ and $\max \{ x_i : Ax \le b\}$. Recall
that this can be done in polynomial time, for example by the ellipsoid
method \cite[see e.g.][]{MR95e:90001,MR2003b:90004}, and denote the time
needed to solve such a linear program by $T_{\mathrm{LP}}(A,b)$. After
computing the quantities $\alpha'_i$ and $\beta'_i$, we can compute all the
quantities $\alpha_i$ and $\beta_i$ in $O(V(A,b))$, all the bounds $l_i$ and
$u_i$ in $O(V(A,b))$, and the pieces composing each function $\phi_i$ in
$O\bigl( \frac{V(A,b)}{\log_2(1 + 4\epsilon + 4\epsilon^2)} \bigr)$ per
function. The total running time is therefore $O\bigl( nT_{\mathrm{LP}}(A,b)
+ V(A,b) + \frac{nV(A,b)}{\log_2(1 + 4\epsilon + 4\epsilon^2)} \bigr) =
O\bigl( nT_{\mathrm{LP}}(A,b) + \frac{nV(A,b)}{\epsilon} \bigr)$.

In many applications, the domain $Y$ of the objective function has a very
simple structure and the quantities $\alpha'_i$ and $\beta'_i$ are included
in the input, as part of the description of the objective function. In this
case, using bound \eqref{eq:bnd:poly-gen} directly may yield significant
advantages over the approach that lets $Y := X$ and solves $2n$ linear
programs. Bound \eqref{eq:bnd:poly-gen} can be improved to $4 + 2\big\lceil
\frac{2V(A,b) + 2\size(\alpha'_i) + 2\size(\beta'_i)} {\log_2(1 + 4\epsilon
  + 4\epsilon^2)} \bigr\rceil$, and as $\epsilon \to 0$ it behaves as
$\frac{V(A,b) + \size(\alpha'_i) + \size(\beta'_i)}{(\log_2 e) \epsilon}$.
Since $\alpha'_i$ and $\beta'_i$ are part of the input, the improved bound
is again polynomial in the size of the input and linear in $1/\epsilon$.
\section{Algorithms for Concave Cost Problems}
\label{sect:algor}
Although concave cost problem \eqref{mp:conc:poly} can be approximated
efficiently by piecewise-linear cost problem \eqref{mp:pwl:poly}, both the
original and the resulting problems contain the set cover problem as a
special case, and therefore are NP-hard. Moreover, the set cover problem
does not have an approximation algorithm with a certain logarithmic factor,
unless $\rmP = \rmNP$ \cite{MR1715654}. Therefore, assuming that $\rmP \neq
\rmNP$, we cannot develop a polynomial-time exact algorithm or constant
factor approximation algorithm for problem \eqref{mp:pwl:poly} in the
general case, and then use it to approximately solve problem
\eqref{mp:conc:poly}. In this section, we show how to use our
piecewise-linear approximation approach to obtain new algorithms for concave
cost problems.

We begin by writing problem \eqref{mp:pwl:poly} as an integer program.
Several classical methods for representing a piecewise-linear function as
part of an integer program introduce a binary variable for each piece and
add one or more coupling constraints to ensure that any feasible solution
uses at most one piece \cite[see e.g.][]{MR2000c:90001,Croxton:2003:CMP}.
However, since the objective function of problem \eqref{mp:pwl:poly} is also
concave, the coupling constraints are unnecessary, and we can employ the
following fixed charge formulation. This formulation has been known since at
least the 1960s \cite[e.g.][]{Feldman:1966:WLC}.
\begin{subequations}
\begin{align}
\min\ &\sum_{i=1}^n \sum_{p=0}^{P_i} 
\left(f_i^p z_i^p + s_i^p y_i^p\right),\\
\text{s.t.}\ &Ax \le b, \\
& x_i = \sum_{p=0}^{P_i} y_i^p,  &i\in [n], \\
& 0\le y_i^p \le B_i z_i^p, &i\in [n], p\in \{0,\dots,P_i\}, \\
& z_i^p \in \{0,1\}, &i\in [n], p\in \{0,\dots,P_i\}.
\end{align}
\label{mp:ip:poly}
\end{subequations}
%
% hskip needed, bug somewhere in latex?
\hskip-1ex 
Here, we assume without loss of generality that $\psi_i(0)=0$.
The coefficients $B_i$ are chosen so that $x_i \le B_i$ at any vertex of the
feasible polyhedron $X$ of problem \eqref{mp:pwl:poly}, for instance $B_i =
2^{U(A,b)-1} - 1$.

A key advantage of formulation \eqref{mp:ip:poly} is that, in many cases, it
preserves the special structure of the original concave cost problem. For
example, when \eqref{mp:conc:poly} is the concave cost multicommodity flow
problem, \eqref{mp:ip:poly} becomes the fixed charge multicommodity flow
problem, and when \eqref{mp:conc:poly} is the concave cost facility location
problem, \eqref{mp:ip:poly} becomes the classical facility location problem.
In such cases, \eqref{mp:ip:poly} is a well-studied discrete optimization
problem and may have a polynomial-time exact algorithm, fully
polynomial-time approximation scheme (FPTAS), polynomial-time approximation
scheme (PTAS), approximation algorithm, or polynomial-time heuristic.

Let $\gamma \ge 1$. The next lemma follows directly from Lemma
\ref{lm:poly}.
\begin{lemma}
Let $x'$ be a $\gamma$-approximate solution to problem \eqref{mp:pwl:poly},
that is $x'\in X$ and $Z^*_{\ref{mp:pwl:poly}} \le \psi(x') \le \gamma
Z^*_{\ref{mp:pwl:poly}}$. Then $x'$ is also a $(1+\epsilon)\gamma$
approximate solution to problem \eqref{mp:conc:poly}, that is
$Z^*_{\ref{mp:conc:poly}} \le \phi(x') \le (1+\epsilon) \gamma
Z^*_{\ref{mp:conc:poly}}$.
\end{lemma}

Therefore, a $\gamma$-approximation algorithm for the resulting discrete
optimization problem yields a $(1+\epsilon)\gamma$ approximation algorithm
for the original concave cost problem. More specifically, we compute a
$1+\epsilon$ piecewise-linear approximation of the concave cost problem; the
time needed for the computation and the input size of the resulting problem
are both bounded by $O\bigl( \frac{n U(A,b)}{\epsilon} \bigr)$. Then, we run
the $\gamma$-approximation algorithm on the resulting problem. The following
table summarizes the results for other types of algorithms.

\begin{center}
\renewcommand{\arraystretch}{1.25}
\begin{tabular}{ll}
\parbox[][6ex][c]{0.4\textwidth}{When the resulting discrete \\
optimization problem has a ...}  &
\parbox[][6ex][c]{0.35\textwidth}{We can obtain for the original \\
concave cost problem a ...} \\
\hline
\hline
Polynomial-time exact algorithm & 
FPTAS \\
FPTAS & 
FPTAS \\
PTAS &
PTAS \\
$\gamma$-approximation algorithm &
$(1+\epsilon)\gamma$ approximation algorithm \\
Polynomial-time heuristic &
Polynomial-time heuristic \\
\hline
\hline
\end{tabular}
\end{center}
In conclusion, we note that the results in this section can be adapted to
the more general problems \eqref{mp:conc:poly-gen} and
\eqref{mp:pwl:poly-gen}.
\section{Concave Cost Multicommodity Flow}
\label{sect:flow}
To illustrate our approach on a practical problem, we consider the concave
cost multicommodity flow problem. Let $(V,E)$ be an undirected network with
node set $V$ and edge set $E$, and let $n = |V|$ and $m = |E|$. This network
has $K$ commodities flowing on it, with the supply or demand of commodity
$k$ at node $i$ being $b^k_i$. If $b^k_i > 0$ then node $i$ is a source for
commodity $k$, while $b^k_i < 0$ indicates a sink. We assume that each
commodity has one source and one sink, that the supply and demand for each
commodity are balanced, and that the network is connected.

Each edge $\{i,j\} \in E$ has an associated nondecreasing concave cost
function $\phi_{ij} : \mathbb{R}_+ \to \mathbb{R}_+$. Without loss of
generality, we let $\phi_{ij}(0) = 0$ for $\{i,j\} \in E$. For an edge $\{i,
j\} \in E$, let $x^k_{ij}$ indicate the flow of commodity $k$ from $i$ to
$j$, and $x^k_{ji}$ the flow in the opposite direction. The cost on edge
$\{i,j\}$ is a function of the total flow of all commodities on it, namely
$\phi_{ij}\bigl( \sum_{k=1}^K (x^k_{ij} + x^k_{ji}) \bigr)$. The goal is to
route the flow of each commodity so as to satisfy all supply and demand
constraints, while minimizing total cost.

A mathematical programming formulation for this problem is given by:
\begin{subequations}
\label{mp:conc:flow}
\begin{align}
Z^*_{\ref{mp:conc:flow}} = 
\min\ &\sum_{\{i,j\} \in E} \phi_{ij}\left(\sum_{k=1}^K 
(x_{ij}^k+x_{ji}^k) \right),\\
\text{s.t.}\ &\sum_{\{i, j\}\in E} x_{ij}^k - 
\sum_{\{j,i\} \in E} x_{ji}^k = b_i^k, &i\in V, k\in [K],\\
&x_{ij}^k,x_{ji}^k \ge 0, &\{i,j\} \in E, k\in [K].
\end{align}
\end{subequations}
Let $B^k = \sum_{i: b_i^k>0} b_i^k$ and $B = \sum_{k=1}^K B^k$. For
simplicity, we assume that the coefficients $b^k_i$ are integral.

A survey on concave cost network flows and their applications is available
in \cite{MR91k:90211}. Concave cost multicommodity flow is also known as the
buy-at-bulk network design problem \cite[e.g][]{MR2181615,Chekuri:2006:AAN}.
Concave cost multicommodity flow has the Steiner tree problem as a special
case, and therefore is NP-hard, and does not have a polynomial-time
approximation scheme, unless $\rmP = \rmNP$ \cite{MR1015826,MR1639346}.
Moreover, concave cost multicommodity flow does not have an
$O\bigl(\log^{1/2 - \epsilon'} n \bigr)$ approximation algorithm for
$\epsilon'$ arbitrarily close to $0$, unless $\rmNP \subseteq \ZTIME\bigl(
\eta^{\mathrm{polylog} \eta} \bigr)$ \cite{Andrews:2004:HBN}.

Problem \eqref{mp:conc:flow} satisfies the assumptions needed by Lemma
\ref{lm:poly}, since we can handle the cost functions $\phi_{ij}\bigl(
\sum_{k=1}^K (x^k_{ij} + x^k_{ji}) \bigr)$ by introducing new variables
$\xi_{ij} = \sum_{k=1}^K (x^k_{ij} + x^k_{ji})$ for $\{i,j\} \in E$. We
apply the approach of equations \eqref{eq:poly:lu}--\eqref{eq:poly:bound} to
approximate this problem to within a factor of $1+\epsilon$, use formulation
\eqref{mp:ip:poly} to write the resulting problem as an integer program, and
disaggregate the integer program, obtaining:
\begin{subequations}
\label{mp:ip:flow}
\begin{align}
Z^*_{\ref{mp:ip:flow}} = 
\min\ &\sum_{\{i,j,p\}\in E'} f_{ijp} z_{ijp} + 
\sum_{\{i,j,p\}\in E'} \sum_{k=1}^K s_{ijp} (x_{ijp}^k+x_{jip}^k),\\ 
\text{s.t.}\ &\sum_{\{i,j,p\}\in E'} x_{ijp}^k -
\sum_{\{j,i,p\} \in E'} x_{jip}^k = b_i^k, \hskip 5em i\in V, k\in [K],\\ 
&0 \le x_{ijp}^k,x_{jip}^k \le B^k y_{ijp}, \hskip 8em \{i,j,p\}\in E', 
k\in [K],\\ 
&y_{ijp}\in \{0,1\}, \hskip 14.25em \{i,j,p\}\in E'.
\end{align}
\end{subequations}
This is the well-known fixed charge multicommodity flow problem, but on a
new network $(V, E')$ with $(P+1)m$ edges, for a suitably defined $P$. Each
edge $\{i,j\}$ in the old network corresponds to $P+1$ parallel edges
$\{i,j,p\}$ in the new network, with $p$ being an index to distinguish
between parallel edges. For each edge $\{i,j,p\}\in E'$, the coefficient
$f_{ijp}$ can be interpreted as its installation cost, and $s_{ijp}$ as the
unit cost of routing flow on the edge once installed. The binary variable
$y_{ijp}$ indicates whether edge $\{i,j,p\}$ is installed.

For a survey on fixed charge multicommodity flow, see
\cite{MR1420866,MR99j:90001-i}. This problem is also known as the
uncapacitated network design problem. The above hardness results from
\cite{MR1015826,MR1639346} and \cite{Andrews:2004:HBN} also apply to fixed
charge multicommodity flow.

By bound \eqref{eq:poly:bound}, $P \le \bigl\lceil \frac{2U(A,b')}{\log_2 (1
  + 4\epsilon + 4\epsilon^2)} \bigr\rceil$, with $A$ and $b'$ being the
constraint matrix and right-hand side vector of problem
\eqref{mp:conc:flow}. However, we can obtain a much lower value of $P$ by
taking problem structure into account. Specifically, we perform the
approximation with $l_i = 1$ and $u_i = B$, which results in $P \le
\bigl\lceil \frac{\log_2 B}{\log_2 (1 + 4\epsilon + 4\epsilon^2)}
\bigr\rceil$.

\begin{lemma}
$Z^*_{\ref{mp:conc:flow}} \le Z^*_{\ref{mp:ip:flow}} \le (1+\epsilon)
  Z^*_{\ref{mp:conc:flow}}$.
\label{lm:flow}
\end{lemma}
\begin{proof}
Since the objective is concave and nonnegative, problem (\ref{mp:conc:flow})
has an optimal solution at a vertex $z$ of its feasible polyhedron
\cite[][]{MR0131816}. In $z$, the flow of each commodity occurs on a tree
\cite[see e.g.][]{MR1956924}, and therefore, the total flow $\sum_{k=1}^K
(z^k_{ij} + z^k_{ji})$ on any edge $\{i,j\} \in E$ is in $\{ 0 \} \cup
     [1,B]$. The approximation result follows from Theorem
     \ref{th:gen-tight}.
\end{proof}
\subsection{Computational Results}
\label{sect:flow:comput}
We present computational results for problems with complete uniform
demand---there is a commodity for every ordered pair of nodes, and every
commodity has a demand of 1. We have generated the instances based on
\cite{MR90h:90061} as follows. To ensure feasibility, for each problem we
first generated a random spanning tree. Then we added the desired number of
edges between nodes selected uniformly at random. For each number of nodes,
we considered a dense network with $n(n-1)/4$ edges (rounded down to the
nearest multiple of 5), and a sparse network with $3n$ edges. For each
network thus generated, we have considered two cost structures.

The first cost structure models moderate economies of scale. We assigned to
each edge $\{i,j\}\in E$ a cost function of the form $\phi_{ij}(\xi_{ij}) =
a+b (\xi_{ij})^c$, with $a,b$, and $c$ randomly generated from uniform
distributions over $[0.1,10], [0.33, 33.4]$, and $[0.8,0.99]$. For an
average cost function from this family, the marginal cost decreases by
approximately 30\% as the flow on an edge increases from 25 to 1,000. The
second cost structure models strong economies of scale. The cost functions
are as in the first case, except that $c$ is sampled from a uniform
distribution over $[0.0099,0.99]$. In this case, for an average cost
function, the marginal cost decreases by approximately 84\% as the flow on
an edge increases from 25 to 1,000. Note that on an undirected network with
$n$ nodes, there is an optimal solution with the flow on each edge in $\{0,
2, \dots, n(n-1) \}$.

Table \ref{tbl:flow:sizes} specifies the problem sizes. Note that although
the individual dimensions of the problems are moderate, the resulting number
of variables is large, since a problem with $n$ nodes and $m$ edges yields
$n(n-1)m$ flow variables. The largest problems we solved have 80 nodes,
1,580 edges, and 6,320 commodities. To approach them with an MIP solver,
these problems would require 1,580 binary variables, 9,985,600 continuous
variables and 10,491,200 constraints, even if we replaced the concave
functions by fixed charge costs.

We chose $\epsilon=0.01=1\%$ for the piecewise linear approximation. Here,
we have been able to reduce the number of pieces significantly by using the
tight approximation guarantee of Theorem \ref{th:gen-tight} and the
problem-specific bound of Lemma \ref{lm:flow}. After applying our piecewise
linear approximation approach, we have reduced the number of pieces further
by noting that for low argument values, our approach introduced tangents on
a grid denser than the uniform grid $2, 4, 6, \dots$ For each problem, we
have reduced the number of pieces per cost function by approximately 47 by
using the uniform grid for low argument values, and the grid generated by
our approach elsewhere.
\begin{table}[t]
\begin{center}
\small
\begin{tabular}{rrrrrrrrr}
\# &
$\qquad n$ & $\qquad m$ & $\qquad K$ &
\parbox[][6ex][c]{0.12\textwidth}{\centering Flow Variables} & Pieces \\
\hline
1& 10 & 30 & 90 & 8,100 & 41 \\
2& 20 & 60 & 380 & 22,800 & 77 \\
3& 20 & 95 & 380 & 36,100 & 77 \\
\hline
4& 30 & 90 & 870 & 78,300 & 98  \\
5& 30 & 215 & 870 & 187,050 & 98 \\
\hline
6& 40 & 120 & 1,560 & 187,200 & 113 \\
7& 40 & 390 & 1,560 & 608,400 & 113 \\
\hline
8& 50 & 150 & 2,450 & 367,500 & 124 \\
9& 50 & 610 & 2,450 & 1,494,500 & 124 \\
\hline
10& 60 & 180 & 3,540 & 637,200 & 133 \\
11& 60 & 885 & 3,540 & 3,132,900 & 133 \\
\hline
12& 70 & 210 & 4,830 & 1,014,300 & 141 \\
13& 70 & 1,205 & 4,830 & 5,820,150 & 141 \\
\hline
14& 80 & 240 & 6,320 & 1,516,800 & 148 \\
15& 80 & 1,580 & 6,320 & 9,985,600 & 148 \\
\hline
\end{tabular}
\end{center}
\caption{Network sizes. The column ``Pieces'' indicates the number of pieces
  in each piecewise linear function resulting from the approximation.}
\label{tbl:flow:sizes}
\end{table}

We used an improved version of the dual ascent method described by
Balakrishnan et al. \cite{MR90h:90061} (also known as the primal-dual method
\cite[see e.g.][]{nostd:gw97}) to solve the resulting fixed charge
multicommodity flow problems. The method produces a feasible solution, whose
cost we denote by $Z_{\ref{mp:ip:flow}}^{\mathrm{DA}}$, to problem
(\ref{mp:ip:flow}) and a lower bound $Z_{\ref{mp:ip:flow}}^{\mathrm{LB}}$ on
the optimal value of problem (\ref{mp:ip:flow}). As a result, for this
solution, we obtain an optimality gap $\epsilon_{\mathrm{DA}} =
\frac{Z_{\ref{mp:ip:flow}}^{\mathrm{DA}}}{Z_{\ref{mp:ip:flow}}^{\mathrm{LB}}}-1$
with respect to the piecewise linear problem, and an optimality gap
$\epsilon_{\mathrm{ALL}} = (1+\epsilon)(1+\epsilon_{\mathrm{DA}}) - 1$ with
respect to the original problem.

Table \ref{tbl:flow:comput} summarizes the computational results. We
performed all computations on an Intel Xeon 2.66 GHz. For each problem size
and cost structure, we have averaged the optimality gap, computational time,
and number of edges in the computed solution over 3 randomly-generated
instances.

We obtained average optimality gaps of 3.66\% for problems with moderate
economies of scale, and 4.27\% for problems with strong economies of scale.
This difference in average optimality gap is consistent with computational
experiments in the literature that analyze the difficulty of fixed charge
problems as a function of the ratio of fixed costs to variable costs
\cite{MR90h:90061,MR996585}. Note that the solutions to problems with
moderate economies of scale have more edges than those to problems with
strong economies of scale; in fact, in the latter case, the edges always
form a tree.
\begin{table}[t]
\begin{center}
\small
\begin{tabular}{rrrrrrp{0.04\linewidth}rrrr}
\multirow{2}{*}{\#} && \multicolumn{4}{c}{Moderate economies of scale}
&  & \multicolumn{4}{c}{Strong economies of scale} \\
\cline{3-6} \cline{8-11}
& \qquad & Time & \parbox[][6ex][c]{0.06\textwidth}{\centering\small Sol. Edges}
 & $\epsilon_{\mathrm{DA}}\%$ & $\epsilon_{\mathrm{ALL}}\%$ && Time &
\parbox[][6ex][c]{0.06\textwidth}{\centering\small Sol. Edges}
 & $\epsilon_{\mathrm{DA}}\%$ & $\epsilon_{\mathrm{ALL}}\%$ \\
\hline
1&& 0.13s & 14 & 0.41 & 1.41 && 0.19s & 9 & 0.35 & 1.35 \\
2&& 3.17s & 31 & 1.45 & 2.46 && 3.99s & 19 & 1.06 & 2.07 \\
3&& 3.50s & 25.3 & 1.20 & 2.21 && 5.37s & 19 & 3.38 & 4.42 \\
\hline
4&& 18.5s & 43.7 & 1.94 & 2.96 && 11.7s & 29 & 1.18 & 2.20 \\
5&& 31.3s & 44 & 2.16 & 3.19 && 21.1s & 29 & 3.50 & 4.54 \\
\hline
6&& 1m23s & 61.7 & 2.47 & 3.49 && 27.1s & 39 & 2.20 & 3.22 \\
7&& 2m6s & 59 & 3.24 & 4.28 && 1m11s & 39 & 3.17 & 4.21 \\
\hline
8&& 3m45s & 79 & 2.22 & 3.24 && 1m11s & 49 & 3.42 & 4.46 \\
9&& 6m19s & 74.7 & 3.10 & 4.13 && 2m48s & 49 & 4.22 & 5.26 \\
\hline
10&& 8m45s & 95 & 2.58 & 3.61 && 2m1s & 59 & 3.27 & 4.30 \\
11&& 18m52s & 95.7 & 3.64 & 4.68 && 5m59s & 59 & 4.25 & 5.29 \\
\hline
12&& 16m44s & 101.7 & 2.85 & 3.87 && 2m35s & 69 & 3.77 & 4.81 \\
13&& 39m18s & 115.7 & 4.19 & 5.24 && 9m43s  & 69 & 4.98 & 6.03 \\
\hline
14&& 32m46s & 127.7 & 2.82 & 3.84 && 4m25s & 79 &  4.10 & 5.14 \\
15&& 1h24m & 143 & 5.24 & 6.29 && 15m50s & 79 & 5.67 & 6.73 \\
\hline
\multicolumn{4}{l}{\textbf{Average}} & 2.63 & \textbf{3.66}
&&&& 3.24 & \textbf{4.27}\\
\end{tabular}
\end{center}
\caption{Computational results. The values in column ``Sol. Edges''
  represent the number of edges with positive flow in the obtained
  solutions.}
\label{tbl:flow:comput}
\end{table}

To the best of our knowledge, the literature does not contain exact or
approximate computational results for concave cost multicommodity flow
problems of this size. Bell and Lamar \cite{MR1481977} introduce an exact
branch-and-bound approach for \emph{single-commodity} flows, and perform
computational experiments on networks with up to 20 nodes and 96 edges.
Fontes et al. \cite{MR2003734} propose a heuristic approach based on local
search for single-source single-commodity flows, and present computational
results on networks with up to 50 nodes and 200 edges. They obtain average
optimality gaps of at most 13.81\%, and conjecture that the actual gap
between the obtained solutions and the optimal ones is much smaller. Also
for single-source single-commodity flows, Fontes and Gon\c{c}alves
\cite{MR2334570} propose a heuristic approach that combines local search
with a genetic algorithm, and present computational results on networks with
up to 50 nodes and 200 edges. They obtain optimal solutions for problems
with 19 nodes or less, but do not provide optimality gaps for larger
problems.
\section{Concave Cost Facility Location}
\label{sect:loc}
In the concave cost facility location problem, there are $m$ customers and
$n$ facilities. Each customer $i$ has a demand of $d_i\ge 0$, and needs to
be connected to a facility to satisfy it. Connecting customer $i$ to
facility $j$ incurs a connection cost of $c_{ij} d_i$; the connection costs
$c_{ij}$ are nonnegative and satisfy the metric inequality.

Let $x_{ij} = 1$ if customer $i$ is connected to facility $j$, and $x_{ij} =
0$ otherwise. Then the total demand satisfied by facility $j$ is
$\sum_{i=1}^m d_i x_{ij}$. Each facility $j$ has an associated nondecreasing
concave cost function $\phi_j : \mathbb{R}_+ \to \mathbb{R}_+$. We assume
without loss of generality that $\phi_j(0) = 0$ for $j\in [n]$. At each
facility $j$ we incur a cost of $\phi_j \left( \sum_{i=1}^m d_i c_{ij}
\right)$. The goal is to assign each customer to a facility, while
minimizing the total connection and facility cost.

The concave cost facility location problem can be written as a mathematical
program:
\begin{subequations}
\label{mp:conc:loc}
\begin{align}
Z^*_{\ref{mp:conc:loc}} = 
\min\ &\sum_{j=1}^n \phi_j \left( \sum_{i=1}^m d_i x_{ij}\right)
+ \sum_{j=1}^n \sum_{i=1}^m c_{ij} d_i x_{ij},\\
\text{s.t.}\ &\sum_{j=1}^n x_{ij}=1, \hskip 12em i\in [n],\\
&x_{ij}\ge 0, \hskip 12em i\in [m], j\in [n].
\end{align}
\end{subequations}
Let $D = \sum_{i=1}^m d_i$. We assume that the coefficients $c_{ij}$ and
$d_i$ are integral.

Concave cost facility location has been studied since at least the 1960s
\cite{Kuehn:1963:HPL,Feldman:1966:WLC}. Mahdian and Pal \cite{MR2085471}
developed a $3+\epsilon'$ approximation algorithm for this problem, for any
$\epsilon'>0$. When the problem has unit demands, that is $d_1 = \dots = d_m
= 1$, a wider variety of results become available. In particular, Hajiaghayi
et al. \cite{MR1991628} obtained a 1.861-approximation algorithm. Hajiaghayi
et al. \cite{MR1991628} and Mahdian et al. \cite{MR2247734} described a
1.52-approximation algorithm.

Concerning hardness results, concave cost facility location contains the
classical facility location problem as a special case, and therefore does
not have a polynomial-time approximation scheme, unless $\rmP = \rmNP$, and
does not have a 1.463-approximation algorithm, unless $\rmNP \subseteq
\DTIME\bigl( n^{O(\log \log n)} \bigr)$ \cite{MR1682440}.

As before, problem (\ref{mp:conc:loc}) satisfies the assumptions needed by
Lemma \ref{lm:poly}. We apply the approach of equations
\eqref{eq:poly:lu}--\eqref{eq:poly:bound} to approximate it to within a
factor of $1+\epsilon$, use formulation \eqref{mp:ip:poly} to write the
resulting problem as an integer program, and disaggregate the integer
program:
\begin{subequations}
\label{mp:ip:loc}
\begin{align}
Z^*_{\ref{mp:ip:loc}} = \min\ &\sum_{j=1}^n \sum_{p=0}^P f^p_j y^p_j +
\sum_{j=1}^n \sum_{p=0}^P \sum_{i=1}^m (s_j^p + c_{ij}) d_i x^p_{ij},\\
\text{s.t.}\ &\sum_{j=1}^n \sum_{p=0}^P x^p_{ij}=1, \hskip 10em i\in [m],\\
&0\le x^p_{ij}\le y^p_j, \hskip 6em i\in [m], j\in [n], 
p\in \{0,\dots,P\},\\
&y^p_j\in \{0,1\}, \hskip 8.5em j\in [n], p\in \{0,\dots,P\}.
\end{align}
\end{subequations}
We have obtained a classical facility location problem that has $m$
customers and $Pn$ facilities, with each facility $j$ in the old problem
corresponding to $P$ facilities $(j,p)$ in the new problem. Each coefficient
$f_j^p$ can be viewed as the cost of opening facility $(j,p)$, while $s_j^p
+ c_{ij}$ can be viewed as the unit cost of connecting customer $i$ to
facility $(j,p)$. Note that the new connection costs $s_j^p + c_{ij}$
satisfy the metric inequality. The binary variable $y_j^p$ indicates whether
facility $(j,p)$ is open.

Problem \eqref{mp:ip:loc} is one of the fundamental problems in operations
research \cite{MR1066259,MR2000c:90001}. Hochbaum \cite{MR643582} showed
that the greedy algorithm is a $O(\log n)$ approximation algorithm for it,
even when the connection costs $c_{ij}$ are non-metric. Shmoys et al.
\cite{Shmoys:1997:AAF} gave the first constant-factor approximation
algorithm for this problem, with a factor of 3.16. More recently, Mahdian et
al. \cite{MR2247734} developed a 1.52-approximation algorithm, and Byrka
\cite{Byrka:2007:OBA} obtained a 1.4991-approximation algorithm. The above
hardness results of Guha and Khuller \cite{MR1682440} also apply to this
problem.

Bound \eqref{eq:poly:bound} yields $P \le \bigl\lceil \frac{2U(A,b)}{\log_2
  (1 + 4\epsilon + 4\epsilon^2)} \bigr\rceil$, with $A$ and $b$ being the
constraint matrix and right-hand side vector of problem \eqref{mp:conc:loc}.
We can obtain a lower value for $P$ by taking $l_i = 1$ and $u_i = D$, which
yields $P \le \bigl\lceil \frac{\log_2 D}{\log_2 (1 + 4\epsilon +
  4\epsilon^2)} \bigr\rceil$. The proof of the following lemma is similar
to that of Lemma \ref{lm:flow}.
\begin{lemma}
$Z^*_{\ref{mp:conc:loc}} \le Z^*_{\ref{mp:ip:loc}} \le (1+\epsilon)
  Z^*_{\ref{mp:conc:loc}}$.
\end{lemma}
Combining our piecewise-linear approximation approach with the
1.4991-approximation algorithm of Byrka \cite{Byrka:2007:OBA}, we obtain the
following result.
\begin{theorem}
There exists a $1.4991+\epsilon'$ approximation algorithm for the concave
cost facility location problem, for any $\epsilon'>0$.
\end{theorem}
We can similarly combine our approach with other approximation algorithms
for classical facility location. For example, by combining it with the
1.52-approximation algorithm of Mahdian et al. \cite{MR2247734}, we obtain a
$1.52+\epsilon'$ approximation algorithm for concave cost facility location.
\section*{Acknowledgments}
We are grateful to Professor Anant Balakrishnan for providing us with the
original version of the dual ascent code. This research was supported in
part by the Air Force Office of Scientific Research, Amazon.com, and the
Singapore-MIT Alliance.
%
%
% __________________________________________________________________________
% Back matter
%
\bibliographystyle{alpha}
\bibliography{concave}
\end{document}